\documentclass[12pt,amssymb,verbatim]{amsart}

\usepackage{amsfonts,latexsym,geometry,color,hyperref,ulem}
\usepackage{amssymb,graphics,graphicx}
\usepackage{amsfonts,latexsym,geometry}
\usepackage{amssymb,graphics,graphicx}
\usepackage{bbm}
\usepackage{amscd}
\usepackage{amssymb}
\usepackage{amsthm}
\usepackage[utf8]{inputenc} 
\usepackage[T1]{fontenc}
\usepackage{amsmath}
\usepackage{amsfonts}
\usepackage{amssymb}
\usepackage{graphicx} 
\usepackage{enumerate}
\newcommand\supp{\mathrm{supp}}
\newtheorem{theoreme}{Theorem}[section] %
\newtheorem{proposition}[theoreme]{Proposition} %
\newtheorem{corollary}[theoreme]{Corollary} %
\newtheorem{lemme}[theoreme]{Lemma} %
\newtheorem{definition}{Definition}[section] %
\newtheorem{remark}[theoreme]{Remark} %

\newcommand\sk{\smallskip}

\newcommand\R{\mathbb{R}} 
 
\newcommand\dimm{\underline{\dim}_H}

\RequirePackage{yhmath} 
\renewcommand\widering[1]{\ring{#1}}

\author{E. Daviaud, Uliège}

\begin{document}
\title[Dynamical approximation by mixing systems]{Hausdorff dimension of dynamical Diophantine approximation  associated with ergodic mixing systems}
\maketitle

\begin{abstract}
    In this article, we study the Diophantine approximation by typical orbits with respect to an ergodic  dynamical system which is fast enough mixing with respect to the space $L^1$ and the space of bounded variation function. As an application, of our main result, we are able to fully extend the results proved by Järvenpää, Järvenpää, Myllyoja and Stenflo in the case of random coverings to the dynamical settings. In addition, our result applies on the linear part to the case of ($\times 2,\times 3$, $\mu$) on $\mathbb{T}^2$, where  $\mu$ is a self-affine measure associated with the $(2,3)-$Bedford-McMullen carpet. This result, combined with suitable assumptions regarding the multifractal spectrum of $\mu,$ fully extends a pioneer work of Fan-Schmeling-Troubetzkoy dealing with the case of the doubling of the angle on $\mathbb{T}^1.$ 
\end{abstract}

\section{Introduction}

Given a sequence $(x_n)_{n\in\mathbb{N}}$ of elements of $\mathbb{R}^d$, understanding the dimension of points approximable at a given rate by $(x_n)_{n\in\mathbb{N}}$ has been a topic of interest in Diophantine approximation, dynamical systems and multifractal analysis. Such questions were historically first considered in the context of Diophantine approximation, by Dirichlet and Khintchine, and later on, by Besicovitch and Jarnik (see \cite{khintchine1937,Jarnik}). These authors were interested in estimating, for a given $\psi:\mathbb{N}\to \mathbb{R}_+$, the Hausdorff dimension (which we will denote $\dim_H$) of the sets $$E_{\psi}=\left\{x\in[0,1] \  : \ \vert x-\frac{p}{q}\vert\leq \psi(q)\text{ i.o. }0\leq p\leq q,  \ q\wedge p=1\right\},$$
where i.o. means that the above inequality holds for an infinity of pairs $(p,q).$ It was established by Beresnevich and Velani in \cite{BV} that estimating $\dim_H E_{\psi}$ for every $\psi$ was conditioned to the validity of the so-called Duffin-Schaeffer conjecture, which was proved to hold recently, by Koukoulopoulos and Maynard in  \cite{Maynkou}. As a  consequence,  it is now established that for every mapping $\psi:\mathbb{N}\to \mathbb{R}_+,$ one has $$\dim_H E_{\psi}=\min\left\{1,s_{\psi}\right\}, $$
where $s_{\psi}=\inf\left\{s:  \ \sum_{n\geq 1}\phi(q)\psi(q)^s <+\infty\right\}$ and $\phi$ denotes the Euler mapping. 

Similar questions and problems appear very naturally in dynamical systems and multifractal analysis. In the context of multifractal analysis, the regularity of a given mapping at a  point $x$ often depends on the rate of approximation of $x$ by a specific sequence of points $(x_n)_{n\in\mathbb{N}}$ around which the regularity of $f$ is well understood, see \cite{Jaff-prescription,JaffRiemann} for instance. 

In dynamical system, given $T:X\to X,$ estimating the dimension of sets approximable by a given orbit $(T^n(x))_{n\in\mathbb{N}}$ is also very natural and has been investigated by various authors. Among such study, one can mention the work of Liao and Seuret \cite{LS}, which, refining a pioneer work of Fan-Schmeling-Troubetzkoy \cite{FanSchmTrou}, proved that for any expanding markov map $T:\mathbb{T}^1\to \mathbb{T}^1$ and any Gibbs measure $\mu$ associated with an Hölder potential, for $\mu$-almost every $x$, one has \begin{equation}
\label{thmLS}
\dim_H \left\{y \  : \ \vert T^n(x)-y\vert\leq \frac{1}{n^{\delta}}\text{ i.o. }\right\}=\begin{cases}\frac{1}{\delta}\text{ if }\delta\geq \frac{1}{\dim(\mu)}\\ D_{\mu}(\frac{1}{\delta})\text{ if }\frac{1}{h_0}\leq\delta\leq \frac{1}{\dim(\mu)} \\ 1 \text{ if  }\delta\leq  \frac{1}{h_{0}},\end{cases}
\end{equation}
where $h\mapsto D_{\mu}(h)$ denotes the multifractal spectrum and $h_0$ is such that $D_{\mu}(h_0)=1.$ We also refer to \cite{LiExpAhlf} for the case of ergodic systems $(T,\mu)$ which are exponentially  mixing and the measure $\mu$ is Alfhors-regular.

Later on,  Persson proved  \cite{PerssInhomo} that if an ergodic system of $\mathbb{R}^d$, $(T,\mu),$ mixes exponentially fast (with respect to $L^1$ and the space of bounded variation functions, see Definition \ref{DefDynamicalMix} and Definition \ref{defphimix}), then  for $\mu$-almost every $x$, for every $\delta\geq \frac{1}{\overline{\dim}_H \mu}$
\begin{equation}
\label{equaPersson}
\frac{\kappa_{\mu}}{\delta}\leq\dim_H\left\{y : \ \vert\vert T^{n}(x)-y\vert\vert_{\infty}\leq \frac{1}{n^{\delta}}\text{ i.o. }\right\}\leq \frac{1}{\delta},
\end{equation}
where $\kappa_{\mu}$ depends on the coarse multifractal spectrum of the measure $\mu$. In addition, Persson conjectured that the almost sure  Hausdorff dimension of the above sets should  in general be the same as in the case of approximation by random variables.

In this parallel case, it was recently established by Järvenpää, Järvenpää, Myllyoja and Stenflo in \cite{JJMS} that for almost every  i.i.d. sequence of random variables of common law $\mu \in\mathcal{M}(\mathbb{R}^d)$ $(X_n)_{n\in\mathbb{N}}$, for every $\delta\geq \frac{1}{\overline{\dim}_H \mu},$
one has 

\begin{equation}
\label{resultRandomnonexa}
\dim_H \left\{y : \ \vert\vert X_n -y\vert\vert_{\infty}\leq \frac{1}{n^{\delta}}\text{ i.o. }\right\}=\frac{1}{\delta}.
\end{equation}

In the present article, we establish that \eqref{resultRandomnonexa} also holds in dynamical settings, provided that the ergodic system involved  is super-polynomially mixing (with respect to $L^1$ against the space of bounded variation functions), confirming that the result  conjectured by Persson  holds true. More precisely, under the assumption that $(T,\mu)$ is super-polynomially mixing (Definition \ref{defphimix}), we establish that for every $\delta\geq \frac{1}{\overline{\dim}_H \mu},$ for $\mu$-almost every $x$, one has 
\begin{equation}
\label{mainEquaIntro}
\dim_H\left\{y : \ \vert\vert T^{n}(x)-y\vert\vert_{\infty}\leq \frac{1}{n^{\delta}}\text{ i.o. }\right\}= \frac{1}{\delta}.
\end{equation}
Furthermore, we establish that if $(T,\mu)$ is a $\Sigma$-mixing  (Definition \ref{defphimix}) and $\mu$ is exact-dimensional (Definition \ref{dim}), then,  for every non increasing  sequence $\underline{r}=(r_n)_{n\in\mathbb{N}}\in \mathbb{R}_{+}^{\mathbb{N}}$ such that 
$$s_{\underline{r}}:=\inf\left\{s:\sum_{n\geq 1}r_n^s<+\infty\right\}\leq \dim(\mu),$$
one has $$\dim_H \left\{y : \ \vert\vert T^{n}(x)-y \vert\vert_{\infty}\leq r_n \text{ i.o. }\right\}=s_{\underline{r}}.$$
It is worth mentioning that the random analog of the above result was also established in \cite{JJMS}. In addition, the authors proved that no general formula (involving geometric quantities one usually estimates) holds regarding the dimension of the random covering sets when the measure is not exact-dimensional (or does not admit sets of positive measure on which the local dimension exists as a limit to be  more precise).  

Moreover, in a companion paper of the present article, based an a previous work of Galatolo, Saussol and Rousseau \cite{Saussogalato}, we provide an example of an ergodic system $(T,\mathcal{L}^3)$ on $\mathbb{T}^3$, which is $n \mapsto n^{-s}$-mixing, where $0<s<1$ and for every $x \in\mathbb{T}^3,$ $$\dim_H \left\{y \in\mathbb{T}^3 : \ \vert\vert y-T^{n}(x)\vert\vert_{\infty}\leq \frac{1}{n^{\frac{1}{3}}}\text{ i.o. }\right\}<3.$$
This example shows that our result in the case where $\mu$ is exact-dimensional is optimal in a strong sense: if $(T,\mu)$ is $n\mapsto n^{-s}$-mixing with $s>1$, then it is $\Sigma$-mixing, hence \eqref{mainEquaIntro} holds true whereas if $(T,\mu)$ is $n\mapsto n^{-s}$-mixing with $0<s<1,$ \eqref{mainEquaIntro} needs not hold in general.

We point out that all the particular cases of mappings treated in \cite{FanSchmTrou,LS,PerssInhomo} are conformal and that the conformality displays a crucial role in the proof techniques used in these articles.
As an illustration   of the result established in the present article, we extend  the result established by Fan-Schmeling-Troubetzkoy \cite{FanSchmTrou} in the case of the $\times 2$ map on $\mathbb{T}^1$ to the case of the $\times 2,\times 3$ map, $T_{2,3},$ on $\mathbb{T}^2$ in the $``$ linear part $"$ and, under suitable assumption regarding the multifractal part, we show that the same result holds. More precisely, let $S=\left\{f_1,....,f_m\right\}$ be the canonical IFS associated with  $\mathbb{T}^2$ seen has a $(2,3)$-Bedford-McMullen carpet, $(p_1,...,p_m)\in (0,1)^m$ and  $\mu$ the corresponding self-affine measure. We establish that, for $\mu$-almost every $x$, 
\begin{equation}
\label{equFanShmelTroub}
\dim_H \left\{y \  : \ \vert\vert T_{2,3}^n(x)-y\vert\vert_{\infty}\leq \frac{1}{n^{\delta}}\text{ i.o. }\right\}=\begin{cases}\frac{1}{\delta}\text{ if }\delta\geq \frac{1}{\dim(\mu)}\\ D_{\mu}(\frac{1}{\delta})\text{ if }\frac{1}{h_0}\leq\delta\leq \frac{1}{\dim(\mu)}\text{ and }\mu\text{ is m.f.} \\ \dim_H K \text{ if  }\delta\leq  \frac{1}{h_{0}}\text{ and }\mu\text{ is m.f.},\end{cases}
\end{equation}
where $K=\supp(\mu),$   $h_0$ is such that $D_{\mu}(h_0)=\dim_H K$ and $``$ m.f. $"$ means that the measure $\mu$ is $``$ multifractal regular $"$, as defined in Definition \ref{defmf}. A reader which is well-versed in the dimension theory of self-affine sets and measures will notice that it is known in general that self-affine measure on Bedford-McMullen carpet do not satisfy  the multifractal formalism. However, it turns out that the proof technique (for the corresponding part of the above spectrum) presented in \cite{SeurMarkov} (and slightly generalized in \cite{EP}) is robust enough to deal with that case where $\mu$ is multifractal regular. We also add that there are currently no example of self-affine measure which does not satisfy Definition \ref{defmf} (i.e. is not multifractal regular) and, in view of recent works of Chamberlain-Cann \cite{MultiBedCann}, it seems reasonable to conjecture that every self-affine measure on Bedford-McMullen carpets satisfy this assumption. As, technically, the estimation for $\delta\leq \frac{1}{\dim_H \mu}$ mainly relies on the argument used in \cite{SeurMarkov,EP} the main novelty brought by this article is the case $\delta \geq \frac{1}{\dim(\mu)}.$



\section{Preliminaries, mixing properties of ergodic systems}

Let us start with some notations 

 Let $d$ $\in\mathbb{N}$. For $x\in\mathbb{R}^{d}$, $r>0$,  $B(x,r)$ stands for the closed ball of ($\mathbb{R}^{d}$,$\parallel \ \ \parallel_{\infty}$) of center $x$ and radius $r$. 
 Given a ball $B$, $\vert B\vert$ stands for the diameter of $B$. For $t\geq 0$, $\delta\in\mathbb{R}$ and $B=B(x,r)$,   $t B$ stands for $B(x,t r)$, i.e. the ball with same center as $B$ and radius multiplied by $t$,   and the  $\delta$-contracted  ball $B^{\delta}$ is  defined by $B^{\delta}=B(x ,r^{\delta})$.
\smallskip

Given a set $E\subset \mathbb{R}^d$, $\widering{E}$ stands for the  interior of the set $E$, $\overline{E}$ its  closure and $\partial E =\overline{E}\setminus \widering{E}$ its boundary. If $E$ is a Borel subset of $\R^d$, its Borel $\sigma$-algebra is denoted by $\mathcal B(E)$.
\smallskip

Given a topological space $X$, the Borel $\sigma$-algebra of $X$ is denoted $\mathcal{B}(X)$ and the space of probability measure on $\mathcal{B}(X)$ is denoted $\mathcal{M}(X).$ 

\sk
 The $d$-dimensional Lebesgue measure on $(\mathbb R^d,\mathcal{B}(\mathbb{R}^d))$ is denoted by 
$\mathcal{L}^d$.
\smallskip

For $\mu \in\mathcal{M}(\R^d)$,   $\supp(\mu)=\left\{x\in \mathbb{R}^d: \ \forall r>0, \ \mu(B(x,r))>0\right\}$ is the topological support of $\mu$.
\smallskip

 Given $E\subset \mathbb{R}^d$, $\dim_{H}(E)$ and $\dim_{P}(E)$ denote respectively  the Hausdorff   and the packing dimension of $E$.
\smallskip

Given a set $S$, $\chi_S$ denotes the indicator function of $S$, i.e., $\chi_S(x)=1$ if $x\in S$ and $\chi_S(x)=0$ otherwise.

 \smallskip

 Given $n\in\mathbb{N},$ $\mathcal{D}_n$ denotes the set of dyadic cubes of generation $n$ and $\mathcal{D}$ the set of all dyadic cubes, i.e.
 \begin{equation*}
\mathcal{D}_{n}=\left\{2^{-n}(k_1,...,k_d)+2^{-n}[0,1)^d, \ (k_1,...,k_d)\in\mathbb{Z}^d\right\}\text{ and }\mathcal{D}=\bigcup_{n\geq 0}\mathcal{D}_{n}.
\end{equation*}

\subsection{Recall on geometric measure theory}

\subsubsection{Hausdorff content, measure and dimension}

\begin{definition}
\label{hausgau}
Let $\zeta :\mathbb{R}^{+}\mapsto\mathbb{R}^+$. Suppose that $\zeta$ is increasing in a neighborhood of $0$ and $\zeta (0)=0$. The  Hausdorff outer measure at scale $t\in(0,+\infty]$ associated with the gauge $\zeta$ of a set $E$ is defined by 
\begin{equation}
\label{gaug}
\mathcal{H}^{\zeta}_t (E)=\inf \left\{\sum_{n\in\mathbb{N}}\zeta (\vert B_n\vert) : \, \vert B_n \vert \leq t, \ B_n \text{ closed ball and } E\subset \bigcup_{n\in \mathbb{N}}B_n\right\}.
\end{equation}
The Hausdorff measure associated with $\zeta$ of a set $E$ is defined by 
\begin{equation}
\mathcal{H}^{\zeta} (E)=\lim_{t\to 0^+}\mathcal{H}^{\zeta}_t (E).
\end{equation}
\end{definition}

For $t\in (0,+\infty]$, $s\geq 0$ and $\zeta:x\mapsto x^s$, one simply uses the usual notation $\mathcal{H}^{\zeta}_t (E)=\mathcal{H}^{s}_t (E)$ and $\mathcal{H}^{\zeta} (E)=\mathcal{H}^{s} (E)$, and these measures are called $s$-dimensional Hausdorff outer measure at scale $t\in(0,+\infty]$ and  $s$-dimensional Hausdorff measure respectively. Thus, 
\begin{equation}
\label{hcont}
\mathcal{H}^{s}_{t}(E)=\inf \left\{\sum_{n\in\mathbb{N}}\vert B_n\vert^s : \, \vert B_n \vert \leq t, \ B_n \text{  closed ball and } E\subset \bigcup_{n\in \mathbb{N}}B_n\right\}. 
\end{equation}
The quantity $\mathcal{H}^{s}_{\infty}(E)$ (obtained for $t=+\infty$) is called the $s$-dimensional Hausdorff content of the set $E$.
\begin{definition} 
\label{dim}
Let $\mu\in\mathcal{M}(\mathbb{R}^d)$.  
For $x\in \supp(\mu)$, the lower and upper  local dimensions of $\mu$ at $x$ are  defined as
\begin{align*}
\underline\dim_{{\rm loc}}(\mu,x)=\liminf_{r\rightarrow 0^{+}}\frac{\log(\mu(B(x,r)))}{\log(r)}
 \mbox{ and } \ \    \overline\dim_{{\rm loc}}(\mu,x)=\limsup_{r\rightarrow 0^{+}}\frac{\log (\mu(B(x,r)))}{\log(r)}.
 \end{align*}
Then, the lower and upper Hausdorff dimensions of $\mu$  are defined by 
\begin{equation}
\label{dimmu}
\dimm(\mu)={\mathrm{ess\,inf}}_{\mu}(\underline\dim_{{\rm loc}}(\mu,x))  \ \ \mbox{ and } \ \ \overline{\dim}_P (\mu)={\mathrm{ess\,sup}}_{\mu}(\overline\dim_{{\rm loc}}(\mu,x))
\end{equation}
respectively.
\end{definition}

It is known (for more details see \cite{F}) that
\begin{equation*}
\begin{split}
\dimm(\mu)&=\inf\{\dim_{H}(E):\, E\in\mathcal{B}(\mathbb{R}^d),\, \mu(E)>0\} \\
\overline{\dim}_P (\mu)&=\inf\{\dim_P(E):\, E\in\mathcal{B}(\mathbb{R}^d),\, \mu(E)=1\}.
\end{split}
\end{equation*}
When $\underline \dim_H(\mu)=\overline \dim_P(\mu)$, this common value is simply denoted by $\dim(\mu)$ or $\dim_H \mu$ and~$\mu$ is said to be \textit{exact-dimensional}. It is worth mentioning that many measures which are dynamically defined are exact dimensional. For instance, given a $C^{\infty}$ compact Manifold $M \subset \mathbb{R}^2$ and  $f$, a $C^2$-diffeomorphism  $M\to M$, it was established by Ledrappier and Young in \cite{LY} that any $f$-ergodic measure is exact dimensional.

\subsubsection{Multifractal analysis}


It was first established by Fan, Schmeling and Troubetzkoy that, in the case of the doubling of the angle equipped with an equilibrium measure associated with an Hölder potential,  for $\delta<\frac{1}{\overline{\dim}_H \mu}$ the dimension of  
$$ \left\{y : \ \vert\vert T^{n}(x)-y\vert\vert_{\infty}\leq \frac{1}{n^{\delta}}\text{ i.o. }\right\}$$
are related to the multifractal spectrum of these equilibrium measures.

Let $\mu\in\mathcal{M}(\mathbb{R}^d)$ be a probability measure. Given $h\geq 0,$ define 
\begin{equation}
\label{defleveliso}
\begin{cases}
E_{\mu}( h)=\left\{y\in\supp(\mu) : \ \liminf_{r\to 0}\frac{\log \mu (B(y,r))}{\log r}= h\right\} \\
E_{\mu}(\leq h)=\left\{y\in\supp(\mu) : \ \liminf_{r\to 0}\frac{\log \mu (B(y,r))}{\log r}\leq h\right\} \\ E_{\mu}(<h)=\left\{y\in\supp(\mu) : \ \liminf_{r\to 0}\frac{\log \mu (B(y,r))}{\log r}< h\right\}  \\ \widetilde{E}_{\mu}(h)=\left\{y\in\supp(\mu) : \ \lim_{r\to 0}\frac{\log \mu (B(y,r))}{\log r}=h\right\}.
\end{cases}
\end{equation}
The multifractal spectrum $D_{\mu}$ of $\mu$ is the mapping defined for every $h\geq 0$ by 
\begin{equation}
\label{defDmu}
D_{\mu}(h)=\dim_H E_{\mu}(h).
\end{equation} 
We also define 
\begin{equation}
\begin{cases} D_{\mu}^{\leq}(h)=\dim_H E_{\mu}(\leq h) \\ D_{\mu}^{<}(h)=\dim_H E_{\mu}(<h)  .\end{cases}
\end{equation}

%

For many measures that are defined dynamically, it turns out that the multifractal spectrum coincides with the so-called coarse spectrum. Given $n\in\mathbb{N}$, $h\geq 0$ and $\varepsilon>0,$  define $$N_n(h,\varepsilon)=\#\left\{D\in\mathcal{D}_n : \ \vert D\vert^{h+\varepsilon} \leq\mu(D)\leq \vert D \vert^{h-\varepsilon}\right\}.$$

The multifractal coarse spectrum $G_{\mu}$ of $\mu$ is defined as 
\begin{equation}
G_\mu (h)=\lim_{\varepsilon\to 0}\limsup_{n \to+\infty}\frac{\log N_n(h,\varepsilon)}{n \log 2}.
\end{equation}

\subsection{Mixing properties of ergodic systems}
First, we start this section by mentioning that in this article, an ergodic system will always refer to a couple $(T,\mu)$ where $T$ is measurable and $\mu$ is  a probability measure which is ergodic with respect to $T$. In addition, for the sake of simplicity, one will always assume that $T$ is defined on all $\mathbb{R}^d$, but of course every result presented here readily extends to the case where $T$ is defined on a measurable subset of $\mathbb{R}^d$.

We start by recalling the traditional definition of mixing systems with respect to  classes of observable. 
\begin{definition}
\label{DefDynamicalMix}
Let $(T,\mu)$ be an ergodic system on $\mathbb{R}^d$, let $(\mathcal{F}_1,\vert\vert \cdot \vert\vert_{\mathcal{F}_1}),(\mathcal{F}_2,\vert\vert \cdot \vert\vert_{\mathcal{F}_2})$ be two functional metric space sub-spaces of $L^1(\mu)$ and let $\phi:\mathbb{N}\to \mathbb{R}_+$ be a non increasing mapping with $\phi(n)\to0$. We say that $(T,\mu)$ is $\phi$-mixing with respect to $\mathcal{F}_1,\mathcal{F}_2$ if for every $f\in\mathcal{F}_1$ and $g\in\mathcal{F}_2,$ for every $n\in\mathbb{N},$ one has $$\Big\vert \int f\circ T^n gd\mu -\int fd\mu \int gd\mu\Big\vert\leq \vert\vert f\vert\vert_{\mathcal{F}_1} \times \vert \vert g\vert \vert_{\mathcal{F}_2} \times \phi(n).$$
\end{definition}
A large number of ergodic systems (doubling of the angle, cookie cutters, etc...) satisfies the above property with $\mathcal{F}_1 =L^1(\mu)$ and  $\mathcal{F}_2$ the space of bounded variation functions (B.V. in short) or $\mathcal{F}_1=\mathcal{F}_2=\left\{ f : \ f\text{ is Lipschitz }  \right\}.$
A  weaker analog of such properties, which is more suited to the purpose of the present article and was introduced (in a slightly different fashion) in \cite{LLVZ} is the following. 

\begin{definition}
\label{defphimix}
Let $(T,\mu)$ be an ergodic system of $\mathbb{R}^d$ and let $\phi:\mathbb{N}\to \mathbb{R}_+$ be a non-increasing mapping such that $\phi(n)\to 0$. Let $\mathcal{C}_1,\mathcal{C}_2$ be two collection of Borel sets and let $\gamma\geq 1 $ be a real number. We say that $(T,\mu)$ is $\phi$-mixing with respect $(\mathcal{C}_1,\mathcal{C}_2,\gamma)$   if for  $A\in \mathcal{C}_1$ and every  $B \in\mathcal{C}_2$ we have 
\begin{equation}
\label{equamix}
\mu(A\cap T^{-n}(B)) \leq \gamma\mu(A)\times \mu(B)+ \phi(n)\mu(B).
\end{equation}
In addition, we say that $(T,\mu)$ is $\Sigma$-mixing with respect to $(\mathcal{C}_1,\mathcal{C}_2,\gamma)$ if $$\sum_{n\geq 1}\phi(n)<+\infty.$$
In the case where $\mathcal{C}_1=\left\{B(x,r),x\in\mathbb{R}^d,r>0\right\}$, $\mathcal{C}_2=\left\{B \in\mathcal{B}(\mathbb{R}^d)\right\}$ and $\gamma=1,$ we simply say that $(T,\mu)$ is $\phi,\Sigma$-mixing.
\end{definition}
\textbf{In the rest of the article, one will always precise when  a dynamical system satisfies a mixing property   with respect to two classes of mappings. For instance, one will say that $(T,\mu)$ is $\phi$-mixing with respect $L^1$ and B.V.  whereas if one only says that $(T,\mu)$ is  $\phi$-mixing, it will refer  to Definition \ref{defphimix}.}

\begin{remark}
\begin{itemize}
\item[•] In \cite{LLVZ}, the $\phi$-mixing property is defined by asking that the following more natural inequality holds for every $A\in\mathcal{C}_1$ and $B\in\mathcal{C}_2$ 
\begin{equation}
\label{Normalmix}
\Big\vert\mu(A\cap T^{-n}(B))- \mu(A)\times \mu(B)\Big\vert\leq  \phi(n)\mu(B).
\end{equation}
In the present paper, for technical reasons, one will need the slightly more general property given by \eqref{equamix}.\medskip
\item[•] The $\phi$-mixing property with respect to $\mathcal{C}_1=\left\{B(x,r),x\in\mathbb{R}^d,r>0\right\}$, $\mathcal{C}_2=\left\{B \in\mathcal{B}(\mathbb{R}^d)\right\}$ is implied (for every $\gamma \geq 1$) by the $\phi$-mixing of $(T,\mu)$ with respect to B.V. and $L^1(\mu).$   In particular, the $\phi$-mixing property where for every $n\in\mathbb{N}$ $\phi(n)=C \beta^n,$ for some $0<\beta<1$ is implied by the exponential mixing  from the space of function with bounded variation against $L^1.$\medskip
\item[•] If for every $\kappa\in \mathbb{R}$ $\phi(n)n^{\kappa}\to 0$, we say that $\phi$ is super-polynomial.
\end{itemize}
\end{remark}

In \cite{Galadim}, the author investigated and established numerous of interesting results regarding the relation between hitting times and local dimension under the hypothesis that the dynamical system is super-polynomially mixing with respect to Lipschitz observable. Here, based on  arguments that are similar to the arguments used in \cite{Galadim}, we achieve more precise estimates under the hypothesis that our ergodic system is $\Sigma$-mixing with respect to bounded variation function and $L^1(\mu)$ mappings. Given $x,y \in \mathbb{R}^d$ and $r>0,$ let us set 
\begin{equation}
\label{DefTaur}
\tau_{r}(x,y)=\inf\left\{n\geq 0 : \ T^{n}(x)\in B(y,r)\right\}.     
\end{equation}

\begin{proposition}
\label{LocdimMix}
Let $(T,\mu)$ be an ergodic system which is $\Sigma$ mixing. Then for every $y$ such that $0<\underline{\dim}(\mu,y)\leq \overline{\dim}(\mu,y)<+\infty,$ for $\mu$-almost every $x$, $$\lim_{r\to 0+}\frac{\log \tau_{r}(x,y)}{\log \mu(B(y,r))} =1. $$  
\end{proposition}
The proof is done in Section \ref{secproofLocdimMix}.

\section{Statement of the result and preliminaries}
\label{sec-result}

%

\subsection{General theorems}

Before stating our main result, we start by a small lemma.
\begin{lemme}
Let $(T,\mu)$ an ergodic system and $(r_n)_{n\in\mathbb{N}}$ a non increasing sequence of radii. Then there exists $0 \leq s\leq d$ such  that for $\mu$-almost every $x$, $$ dim_H \limsup_{n\to +\infty}B(T^n(x),r_n)=s.$$ 

\end{lemme}

\begin{proof}
For every $n\in\mathbb{N},$ $B(T^{n}(T(x)),r_n)\supset B(T^{n+1}(x),r_{n+1}),$ which yields $$\limsup_{n\to +\infty}B(T^n(T(x)),r_n)\supset \limsup_{n\to +\infty}B(T^n(x),r_n) .$$
Fix $0\leq t \leq d$ and assume  that there exists a Borel set $E_t$ such that $\mu(E_t)>0$ and $M>0$ such that for every $x\in E_t,$ $$\mathcal{H}^t(\limsup_{n\to +\infty}B(T^n(x),r_n))\leq M.$$
Since $x\mapsto \mathcal{H}^t(\limsup_{n\to +\infty}B(T^n(x),r_n))$ increases along orbits, by Poincaré's recurrence theorem, for $\mu$-almost every $x$, $$ \mathcal{H}^t(\limsup_{n\to +\infty}B(T^n(x),r_n))\leq M.$$
In addition, note that, by Poincaré's formula (regarding union of sets), for every $Q\leq N$, the mapping $$(x,y)\mapsto \chi_{\bigcup_{Q\leq i\leq N}B(T^i(x),r_n)}(y) $$
is measurable and so is $$(x,y)\mapsto\lim_{Q\to+\infty}\lim_{N\to+\infty}\chi_{\cup_{Q\leq i\leq N}B(T^i(x),r_n)}(y)=\chi_{\limsup_{n\to +\infty}B(T^n(x),r_n)}(y).$$
Thus,  the mapping $(x,y)\mapsto \mathcal{H}^t(\limsup_{n\to +\infty}B(T^n(x),r_n))\in L^1(\mu)$ and is increasing along orbits. We conclude  that, by ergodicity, it is almost surely constant. Noticing that either $t$ is in the case mentioned above, either almost surely
$$\mathcal{H}^s(\limsup_{n\to +\infty}B(T^n(x),r_n))=+\infty,$$ we can define $$s=\inf\left\{t: \int \mathcal{H}^t(\limsup_{n\to +\infty}B(T^n(x),r_n))d\mu(x)<+\infty\right\}.$$
Writing $E_s=\bigcap_{t\in\mathbb{Q}\cap(s,d)}F_t$ yields the claim.

\end{proof}

As mentioned in the introduction, it was established in \cite{JJMS} that, in the case of random coverings, when the measure $\mu$ is exact dimensional, then the almost sure  dimension of the corresponding random limsup set can be computed in terms of a critical exponent. We establish here the same result in the case of $\Sigma$-mixing systems.

Our main result is the following one.
\begin{theoreme}
\label{mainthm}
Let $(T,\mu)$ be a $\Sigma$-mixing ergodic  system such that $\mu$ is exact-dimensional. Let  $\underline{r}=(r_n)_{n\in\mathbb{N}}\in\mathbb{R}_+^{\mathbb{N}}$ be a non increasing sequence such that
\begin{equation}
\label{defsr}
s_{\underline{r}}:=\inf\left\{s:\sum_{n\geq 1}r_n^s<+\infty\right\} \leq \dim \mu.
\end{equation}
 
Then, writing for $x\in X$ 
\begin{equation}
\label{defExr}
E(x,\underline{r})=\left\{y:\vert\vert y-T^n(x) \vert\vert_{\infty}\leq r_n \text{ i.o. }\right\},
\end{equation}
one has for $\mu$-almost every $x$, $$\dim_H E(x,\underline{r})=s_{\underline{r}}.$$

\end{theoreme} 

\begin{remark}
\begin{itemize}
\item[•] In the case where $s_{\underline{r}}>\alpha$, one expect  that the for typical $x$,  $$\dim_H E(x,\underline{r})>\dim \mu$$
but, as shown in \cite{FanSchmTrou,LS}, the formula (if it exists) probably involves the multifractal spectrum of $\mu$ and has to be dealt with in a different fashion. \medskip
\item[•] If $\mu$ is  absolutely continuous with respect to the Lebesgue measure, then for every non-increasing sequence of radii $(r_n)_{n\in\mathbb{N}}$, for $\mu$-almost every $x$, $$\dim_H E(x,\underline{r})=\min\left\{d,s_{\underline{r}}\right\}.$$
\end{itemize}
\end{remark}
Given a non increasing sequence of radii $(r_n)_{n\in\mathbb{N}}$, even in the case of random approximation, i.e. when one consider a sequence of i.i.d. random variables $(X_n)_{n\in\mathbb{N}}$ instead of orbits $(T^n(x))_{n\in\mathbb{N}}$, it was proved in \cite{JJMS} that, in the case where $\mu$ is not exact-dimensional, no general formula (involving only quantities one usually estimates) holds regarding the Hausdorff dimension of the sets $$\dim_H \left\{y : \ d(X_n,y)\leq r_n \text{ i.o. }\right\}.$$
However, in the specific case where $(r_n =\frac{1}{n^{\delta}})_{n\in\mathbb{N}}$ with $\delta \geq  \frac{1}{\overline{\dim}_H \mu},$ the authors were able to prove that the same estimates as in the exact-dimensional case still holds. Our second result shows that this result remains true in the case of super-polynomially mixing ergodic systems.
\begin{theoreme}
\label{MainthmNonexa}
Let $(T,\mu)$ be an ergodic system such that $\overline{\dim}_H \mu >0.$ If $(T,\mu)$ is super-polynomially mixing with respect to $(\left\{B(x,r),x \in  \supp(\mu),r>0\right\},\mathcal{B}(\mathbb{R}^d),1),$  then, for every $\delta\geq \frac{1}{\overline{\dim}_H \mu},$ for $\mu$-almost every $x$, one has \begin{equation}
\dim_H \left\{y : \ \vert\vert T^{n}(x)-y\vert\vert_{\infty}\leq \frac{1}{n^{\delta}}\text{ i.o. }\right\}=\frac{1}{\delta}.
\end{equation}
\end{theoreme}

\begin{remark}
If $(T,\mu)$ is super-polynomially mixing from B.V. against $L^1(\mu)$ and $\overline{\dim}_H \mu >0$, then $(T,\mu)$ satisfies the conditions of the above theorem. In addition, with a little more work, one could show that the  Theorem \ref{MainthmNonexa} actually holds under the slightly weaker assumption that there exists $\gamma\geq 1,$ a measurable set $E$ with $\mu(E)=1$ and a measurable mapping $\rho:\mathbb{R}^d \to \mathbb{R}_+$ such that, writing $\mathcal{C}_1 =\left\{B(y,r):0<r\leq \rho(y),y\in E\right\},$  $(T,\mu)$  is $(\mathcal{C}_1,\mathcal{B}(\mathbb{R}^d),\gamma)$-super-polynomially mixing. 
\end{remark}

In the case where $\delta<\frac{1}{\overline{\dim}_H \mu},$ the proof done in the random case in \cite{EP} together with Proposition \ref{LocdimMix} allows one to extend \cite[Proposition 2]{EP}. 

Given a mapping $g:\mathbb{R}\to \mathbb{R}$ let us denote $\widehat{g}$ the mapping defined for every $x$ by $$\widehat{g}(x)=\inf\left\{h(x),h\geq g \text{ and }  h \text{ is 1-Lipschitz}\right\}.$$

 \begin{proposition}
 \label{thmmultifract}
 Let $(T,\mu)$ be an ergodic $\Sigma$-mixing system. Then, for every $0<\delta<\frac{1}{\overline{\dim}_H \mu}$, one has $$\lim_{t\to \frac{1}{\delta}^{-}}D_{\mu}^{\leq}(t)\leq\dim_H \left\{y : \ \vert\vert T^n(y)-x\vert\vert_{\infty}\leq \frac{1}{n^\delta}\text{ i.o. }\right\}\leq \max\left\{D^{\leq}_\mu(\frac{1}{\delta}),\widehat{G}_{\mu}(\frac{1}{\delta})\right\}.$$ 
 \end{proposition}
Notice that  both sides coincides as soon as $D_{\mu}^{\leq}(\cdot)$ is concave continuous and $D_{\mu}^{\leq}(\cdot)=G_{\mu}(\cdot).$  Such condition is for instance met for self-similar measure satisfying the so-called open set condition \cite{Olsen98}. 

 We isolate here the following result regarding the upper-bound, which follows directly from the proof technique used to establish \cite[Proposition 2]{EP}. 
\begin{proposition}
\label{PropoMultiBed}
 Let $\mu \in \mathcal{M}(\mathbb{R}^d)$ be an exact-dimensional measure. Assume that for every $h\geq 0,$ and $\varepsilon>0,$ there exists $m\in\mathbb{N}$, a constant $C>0$ and a sequence of balls  $(B_{n,k})_{n\in\mathbb{N},k\in Y(h,\varepsilon,n)}$ such that
\smallskip
\begin{itemize}
\item[(1)] for every $n\in\mathbb{N}$, for every  $k\in Y(h,\varepsilon,n),$ one has $$C^{-1}\leq\frac{\vert B_{n,k} \vert}{m^{-n}}\leq C,$$
\item[(2)]  for every $n\in\mathbb{N}$, for every  $k\in Y(h,\varepsilon,n),$ $$\#\left\{k' \in Y(h,\varepsilon,n): \ B_{n,k'}\cap B_{n,k}\neq \emptyset\right\}\leq C,$$
\item[(3)] one has $$E_{\mu}(h)\subset \limsup_{ k \in Y(h,\varepsilon,n),n\to+\infty} C B_{n,k},$$
\item[(4)] Defining $$\overline{G}_{\mu}(h)=\lim_{\varepsilon\to 0}\limsup_{n\to +\infty}\frac{\log \# Y(h,\varepsilon,n)}{m\log n},$$
the mapping $h\mapsto \overline{G}_{\mu}(h)$ is concave and satisfies $G_{\mu}(\dim \mu)=\dim(\mu)$.
\end{itemize}
Then, for every sequence of random variables $(X_n)_{n\in\mathbb{N}}$ of common law $\mu$ (non necessarily independent), almost surely,   for every $\delta< \frac{1}{\dim \mu}$, one has $$\dim_H \limsup_{n\to +\infty}B(X_n,\frac{1}{n^\delta})\cap \bigcup_{h>\frac{1}{\delta}}E_{\mu}(h)\leq \overline{G}_{\mu}(\frac{1}{\delta}).$$
\end{proposition} 
 In view of  Proposition \ref{PropoMultiBed}, we introduce the following definition.
 
 \begin{definition}
 \label{defmf}
 Let $\mu \in\mathcal{M}(\mathbb{R}^d)$ be an exact-dimensional measure such that $\sup_{h : \ E_{\mu}(h)\neq \emptyset}h<+\infty$. The measure $\mu$ is said to be multifractal regular (m.f. in short) if it satisfies  items $(1),(2),(3),(4)$ of Proposition \ref{PropoMultiBed} with $\overline{G}_{\mu}$ such that $\overline{G}_{\mu}(h)=D_{\mu}(h)$  for every $h\leq h_c:=\inf\left\{h'  : \ D_{\mu}^{\leq }(h')=\dim_H \supp(\mu)\right\}$. 
 \end{definition}


\subsection{Application to $\times m,\times n$ and self-affine measures on Bedford-McMullen carpets}
Let us first recall the definition of Bedford-McMullen carpets. Let $1<m<n$ be two integers and let $\mathcal{I}\subset \left\{1,...,m\right\}\times \left\{1,...,n\right\}$ and given $(i,j)\in\mathcal{I},$ let $f_{(i,j)}$ be the canonical contraction of $[0,1]^2$ on $(\frac{i}{m},\frac{j}{n})+[\frac{1}{m},\frac{1}{n}].$ Write $S=\left\{f_{(i,j)}\right\}_{(i,j)\in \mathcal{I}}.$

The following classical result is due to Hutchinson.
\begin{theoreme}[\cite{Hutchinson}]
\label{DefBedfordCarpet}
There exists a unique non empty compact set $K_S$ (we call an $(m,n)$-Bedford-McMullen carpet) which satisfies $$K_S =\bigcup_{(i,j)\in\mathbb{I}}f_{(i,j)}(K_S).$$
In addition, given a probability vector $(p_{(i,j)})_{(i,j)\in\mathcal{I}}\in(0,1)^{\# \mathcal{I}}$, there exists a unique probability measure $\mu$ (which supported on $K_S$) satisfying $$\mu(\cdot)=\sum_{(i,j)\in\mathcal{I}}\mu \circ f_{(i,j)}^{-1}(\cdot).$$
\end{theoreme}
Moreover, it is known any such measure $\mu$ is exact-dimensional, as any self-affine measure is \cite{Feng1}.

Let us define $T_{m,n}:\mathbb{T}^2 \to \mathbb{T}^2$ by setting for every $(x,y)\in\mathbb{T}^2,$ $$T_{m,n}(x,y)=(mx,ny).$$

\begin{remark}
One easily establishes that $\mu$ is $T_{m,n}$ ergodic. 
\end{remark}
The next proposition establishes that $(T_{m,n},\mu)$ is $\Sigma$-mixing.
\begin{proposition}
\label{MixingBeford}
    Let $\mu\in\mathcal{M}(\mathbb{R}^d)$ be a self-affine measure associated with an $(m,n)$-Bedford-McMullen carpet. There exists $C>0$ and $\tau<1$ such that for every $n\in\mathbb{N},$ for every ball $A$ and Borel set $B,$ one has $$\mu\Big(T^{-n}(B)\cap A\Big)\leq \mu(A)\mu(B)+C\tau^n \mu(B).$$ 
    In particular, $(T_{m,n},\mu)$ is $\Sigma$-mixing.
\end{proposition}

\begin{proof}
Let $\left\{g_1,...,g_p\right\}=\left\{f_{(i,j)} : \ p_{(i,j)} \neq 0\right\}$ where $(p_{(i,j)})$ is the probability vector defining $\mu$ and $K$ the attractor associated with $\left\{g_1,...,g_p\right\}$. For simplicity, we simply write $T$ rather than $T_{m,n}$ and we  introduce some notations:
\begin{itemize}
    \item we write $\Pi_x,\Pi_y$ the orthogonal projections on the $x$ and $y$ axis,\medskip
    \item  We write $\Lambda=\left\{1,...,p\right\},$ $\Lambda^* =\bigcup_{k\geq 1}\Lambda^k,$ \medskip
    \item given a word $\underline{i}=(i_1,...,i_s)\in\Lambda^s,$ we write $g_{\underline{i}}=g_{i_1}\circ...\circ g_{i_s},$ \medskip
    \item  given a word $\underline{i}=(i_1,...,i_s)\in\Lambda^s,$ we write $R_{\underline{i}}=g_{\underline{i}}([0,1]^2),$ \medskip
    \item  Given $u\in\mathbb{N},$ write $$\mathcal{R}_u =\left\{R_{\underline{i}},\underline{i}\in\Lambda^u\right\}.$$ 
    \item  let $\pi$ be the canonical projection from $\Lambda^{\mathbb{N}}$ to $K$, given for any $(i_n)_{n\in\mathbb{N}}$ by $$\pi((i_n)_{n\in\mathbb{N}})=\lim_{n\to +\infty} g_{i_1}\circ ... \circ g_{i_n}(0).$$
\end{itemize}
Let us recall that both $\mu_x:=\mu \circ \Pi_x^{-1}$ and $\mu_y:=\mu \circ \Pi_y^{-1}$ are self-similar measures satisfying the open set condition supported on the $x$ and $y$ axis. This in particular implies that there exists $\kappa>0$ and $\alpha>0$ such  that, for every interval $I$, one has $$\max\left\{\mu_x(I \times \left\{0\right\}),\mu_y(\left\{0\right\}\times I)\right\}\leq \vert I \vert^{\alpha}.$$
Fix  a ball $A$,  a Borel set $B$ and $k\in\mathbb{N}.$ Notice that $$\pi^{-1}\Big(T^{-k}(B)\Big)=\bigcup_{\underline{i}\in \Lambda^k}\left\{\underline{i}x,x\in \pi^{-1}( B)\right\}.$$
We are going to sort the collection of rectangles of generation $k$ (i.e. the sets $R_{\underline{i}},$ where $\underline{i}\in\Lambda^k$) intersecting the ball $A$ into three families. Recall that $A$ is a square (we use $\vert\vert \cdot \vert\vert_{\infty}$) and define 
\begin{equation}
\begin{cases}
   \mathcal{A}_{1,k}= \left\{R\in\mathcal{R}_k : \ \text{R intersects an horizontal side of }A\right\} \\   \mathcal{A}_{2,k}= \left\{R\in\mathcal{R}_k : \ \text{R intersects a vertical side of }A\right\}   \\ \mathcal{A}_{3,k}=\left\{R\in\mathcal{R}_k : R\cap A \neq \emptyset\right\}\setminus \mathcal{A}_{1,k}\cup \mathcal{A}_{2,k}.
\end{cases}
\end{equation}
The definition of the sets $\mathcal{A}_{1,k},\mathcal{A}_{2,k},\mathcal{A}_{3,k}$ yields that $$\mu\Big(T^{-k}(B)\cap A\Big)\leq\mu\Big(T^{-k}(B)\cap \bigcup_{R \in \mathcal{A}_{1,k}\cup \mathcal{A}_{2,k}\cup \mathcal{A}_{3,k}}R\Big)  .$$
In addition, from the self-affinity equation, one easily gets that, for every $R\in \mathcal{A}_{1,k}\cup \mathcal{A}_{2,k}\cup \mathcal{A}_{3,k},$ one has  $$\mu(T^{-k}(B)\cap R)=\mu(R)\mu(B).$$

Notice that $\bigcup_{R\in \mathcal{A}_{1,k}}R \subset S\cup S',$ where $S=[0,1]\times I,$ $S'=[0,1]\times I',$ and $I,I'$ are intervals satisfying $\vert I \vert=\vert I' \vert=\frac{2}{n^k}.$ Hence we have $$\mu\Big(\bigcup_{R\in \mathcal{A}_{1,k}}R \Big)\leq \mu(S)+\mu(S')=\mu_y(I)+\mu_y(I') \leq 2\kappa\frac{2^\alpha}{n^{\alpha k }}$$
so that $$\mu\Big(T^{-k}(B)\cap \bigcup_{R\in \mathcal{A}_{1,k}}R\Big) \leq \mu(B)\times 2\kappa\frac{2^\alpha}{n^{\alpha k }}.$$
A similar argument shows that $$\mu\Big(T^{-k}(B)\cap \bigcup_{R\in \mathcal{A}_{2,k}}R\Big) \leq \mu(B)\times 2\kappa\frac{2^\alpha}{m^{\alpha k }}.$$
Fix $\underline{i}\in\Lambda^k$ such that $R_{\underline{i}}\in \mathcal{A}_{3,k}$. Notice that, by definition of $\mathcal{A}_{3,k},$ one must have $R_{\underline{i}}\subset A.$ 
Thus one has $$\mu\Big(T^{-k}(B)\cap\bigcup_{R \in\mathcal{A}_{3,k}}R\Big)=\mu(B)\times \mu\Big(\bigcup_{R \in\mathcal{A}_{3,k}}R\Big)\leq \mu(B)\times \mu(A) .$$
These three estimates yields 
\begin{align*}
    \mu(T^{-k}(B)\cap A) \leq \mu(A)\mu(B)+4\kappa 2^{\alpha}\frac{1}{m^{\alpha k}}\mu(B),
\end{align*}
which concludes the proof.
\end{proof}
%

SaussolHandling  properly the $``$ multifractal part $"$ of \eqref{equFanShmelTroub}  would require a precise understanding of $h\mapsto D_{\mu}(h)$.  Unfortunately, it appears that  so far, only the spectrum $h\mapsto \dim_H \widetilde{E}_{\mu}(h)$ has been estimated, by Jordan and Rams \cite{RamsJord} and estimating the $``$ $\liminf$ $"$ spectrum $D_{\mu}$ raises many difficulties and  remains an open problem.   As mentioned in the  introduction, recent work of Chamberlain-Cann showed that a certain class of self-affine measures on Bedford McMullen carpets satisfies a symbolic version of Definition \ref{defmf} (see proof of \cite[Theorem 3.12]{MultiBedCann}) so that, at our current state of knowledge, it seems reasonable to conjecture that every self-affine measure on Bedford-McMullen carpet are multifractal regular (Definition \ref{defmf}). With such an assumption, Proposition \ref{PropoMultiBed} would yield the following corollary.

\begin{corollary}
Let $K_S$ be an $(m,n)$-Bedford-McMullen carpet and $\mu$ a self-affine measure associated with $S.$ Assume that $\mu$ is multifractal regular,  then for every $0<\delta<\frac{1}{\dim(\mu)},$ for $\mu$-almost every $x\in\mathbb{T}^2,$ one has $$\dim_H \left\{y : \ \vert\vert T^k_{m,n}(x)-y \vert\vert_{\infty}\leq \frac{1}{n^\delta}\text{ i.o. }\right\}=D_{\mu}^{\leq}(h).$$ 
\end{corollary}

Combining the above corollary with Theorem \ref{mainthm}, one obtains that \eqref{thmLS} holds in the case of self-affine measures on Bedford-McMullen carpets assuming that these measures are multifractal regular.

\begin{corollary}
Let $\mu \in\mathcal{M}(\mathbb{T}^2)$ be a self-affine measure associated with an $(m,n)$-Bedford-McMullen carpet. Then for $\mu$-almost every $x$, one has 
\begin{equation}
\dim_H \left\{y \  : \ \vert\vert T_{m,n}^k(x)-y\vert\vert_{\infty}\leq \frac{1}{k^{\delta}}\text{ i.o. }\right\}=\begin{cases}\frac{1}{\delta}\text{ if }\delta\geq \frac{1}{\dim(\mu)}\\ D_{\mu}(\frac{1}{\delta})\text{ if }\frac{1}{h_0}\leq\delta\leq \frac{1}{\dim(\mu)}\text{ and }\mu \text{ is m.f.}\\ \dim_H K \text{ if  }\delta\leq  \frac{1}{h_{0}}\text{ and }\mu \text{ is m.f.}\end{cases}
\end{equation}
\end{corollary}

\section{Proof of Proposition \ref{LocdimMix}}
\label{secproofLocdimMix}

\begin{proof}
We write $\alpha=\underline{\dim}(\mu,y).$ First notice that, for every $x,y$ one has, for every $\eta>0,$ 

\begin{equation}
\begin{cases}\liminf_{n\to +\infty}\frac{\log(\tau_{\frac{1}{n^{\eta}}}(x,y))}{-\log \mu(B(y,\frac{1}{n^\eta}))}=\liminf_{k\to+\infty}\frac{\log(\tau_{2^{-k}}(x,y))}{-\log \mu(B(y,2^{-k}))}=\liminf_{r\to 0}\frac{\log(\tau_{r}(x,y))}{-\log\mu(B(y,r))} \\ \\
\limsup_{n\to +\infty}\frac{\log(\tau_{\frac{1}{n^{\eta}}}(x,y))}{-\log \mu(B(y,\frac{1}{n^\eta}))}=\limsup_{k\to+\infty}\frac{\log(\tau_{2^{-k}}(x,y))}{-\log \mu(B(y,2^{-k}))}=\limsup_{r\to 0}\frac{\log(\tau_{r}(x,y))}{-\log\mu(B(y,r))}.
\end{cases}
\end{equation}

and

Let us first recall the following result.

\begin{theoreme}[\cite{Galadim}, Lemma 3]
For every $\gamma>\frac{1}{\alpha}$, for $\mu$-almost every $x$, $$\liminf_{n\to+\infty}n^{\gamma}\min_{1\leq i\leq n}\vert\vert T^i(x)-y\vert\vert_{\infty}=+\infty.$$
\end{theoreme}
As a direct consequence, for every $\gamma>\alpha$, for $\mu$-almost every $x$, there exists $N$ so large that, for every $n\geq N,$ one has, for every $1\leq i\leq n,$ 
$$\vert\vert T^i(x)-y\vert\vert_{\infty}\geq \frac{1}{n^{\gamma}}.$$
This yields that $$\tau_{\frac{1}{n^{\gamma}}}(x,y)\geq n=\Big(\frac{1}{n^{\gamma}}\Big)^{\frac{-1}{\gamma}}\geq \mu\Big(B(y,\frac{1}{n^{\gamma}})\Big)^{\frac{-1}{\gamma \alpha}}.$$
In particular, taking $\gamma$ so that $\frac{1}{\gamma\alpha}\geq 1-\varepsilon,$ one gets that, for $\mu$-almost every $x$, 
\begin{equation}
\label{equestitau1}
\liminf_{n\to +\infty}\frac{\log(\tau_{\frac{1}{n^{\gamma}}}(x,y))}{-\log \mu(B(y,\frac{1}{n^\gamma}))}\geq 1-\varepsilon.
\end{equation}

On the other hand, for $r>0,$ we define $N_r \in\mathbb{N}$ an integer that will be precised later on and we set $$E_{N_r}=\bigcup_{1 \leq i\leq N_r}T^{-i}\Big(B(y,r)\Big) .$$ 

Define also $A_i=T^{-i}\Big(B(y,r)\Big).$ Using that $(T,\mu)$ is $\phi$-mixing, one has 
\begin{align*}
&\sum_{1 \leq i,j\leq N_r}\mu(A_i \cap A_j)=N_r \mu(B(y,r))+2\sum_{1\leq i<j\leq N_r}\mu\Big(T^{i-j}(A_i)\cap A_j\Big)\\
&\leq N_r \mu(B(y,r))+2\sum_{1\leq i<j\leq N_r}\phi(j-i)\mu(B(y,r))+N_r(N_r+1)\mu(B(y,r))^2  \\
&\leq N_r \mu(B(y,r))+2\mu(B(y,r))\sum_{1\leq k\leq N_r} (N_r-k)\phi(k)+N_r(N_r+1)\mu(B(y,r))^2\\
&\leq N_r \mu(B(y,r))+2CN_r\mu(B(y,r))+N_r(N_r+1)\mu(B(y,r))^2,
\end{align*}
where $C=\sum_{k\geq 1}\phi(k).$
So by, Chung-Erdös inequality, one gets 
\begin{align*}
\mu\Big(\bigcup_{1\leq i\leq N_r}A_i\Big)\geq \frac{ \Big(N_r \mu(B(y,r))\Big)^2}{N_r \mu(B(y,r))+2CN_r\mu(B(y,r))+N_r(N_r+1)\mu(B(y,r))^2}
\end{align*}
Hence, for $r$ small enough, since $\mu(B(y,r))\leq r^{\alpha-\varepsilon},$ writing $$N_r=\lfloor\frac{1}{\mu\Big(B(y,r)\Big)^{1+\varepsilon}}\rfloor +1,$$ one has
\begin{align*}
1-\mu(E_{N_r})\leq \frac{N_r \mu(B(y,r))+2CN_r\mu(B(y,r))+N_r\mu(B(y,r))^2}{N_r \mu(B(y,r))+2CN_r\mu(B(y,r))+N_r(N_r+1)\mu(B(y,r))^2}
\end{align*}
so that 

\begin{align*}
1-\mu(E_{N_r})\leq  \frac{N_r \mu(B(y,r))}{N_r^2 \mu\Big(B(y,r)\Big)^2}\times \frac{2+2C}{1+2C+2}\leq \kappa \mu\Big(B(y,r)\Big)^{\varepsilon}\leq \kappa r^{\theta_{\varepsilon}},
\end{align*}
where $\theta_{\varepsilon}=(\alpha-\varepsilon)\times \varepsilon.$
In particular, written abusively $E_k=E_{N_{2^{-k}}}$, one has $$\sum_{k\geq 1}(1-\mu(E_{k}))<+\infty.$$
By Borel Cantelli's lemma, we conclude that $$\mu\Big(\liminf_{k\to +\infty}E_k\Big)=1.$$
In  other word, for $\mu$-almost every $x$, there exists $k_x\in\mathbb{N}$ so large that, for every $k \geq k_x,$  $x\in E_{k}$, so that there exists $1 \leq i \leq \frac{2}{\mu\Big(B(y,2^{-k})\Big)^{1+\varepsilon}}$ such that $x\in T^{-i}\Big(B(y,r)\Big).$ This yields 
\begin{equation}
\label{equestitau2}
\tau_{2^{-k}}(x,y)\leq \frac{2}{\mu\Big(B(y,2^{-k})\Big)^{1+\varepsilon}} .
\end{equation}
Combining \eqref{equestitau1} and \eqref{equestitau2}, one obtains that, for $\mu$-almost every $x$, $$1-\varepsilon\leq\liminf_{r\to 0}\frac{\log \tau_r(x,r)}{-\log \mu(B(y,r))}\leq \limsup_{r\to 0}\frac{\log \tau_r(x,r)}{-\log \mu(B(y,r))}\leq 1+\varepsilon.$$
Letting $\varepsilon\to 0$ along a countable sequence yields the result. 
\end{proof}

\section{The case $\delta<\frac{1}{\overline{\dim}_H \mu}$}

First, as mentioned in Section \ref{sec-result}, the upper-bound provided by Proposition \ref{thmmultifract} readily follows from \cite{SeurMarkov,EP}\footnote{We emphasize that in the articles mentioned, the authors prove the result in the case of sequences of independent identically distributed  random variables  but the actual proof does not use the independency of the sequence. Note that this is not so surprising as heuristically, one expects that the additional assumption that the variable are independent can only increase the Hausdorff dimension of the corresponding random limsup set.} So we now established the lower-bound, which is also a simple consequence of Proposition \ref{LocdimMix} and the proof techniques used to establish the lower-bound of \cite[Proposition 15]{EP}.

\begin{lemme}
Let $(T,\mu)$ be a $\Sigma$-mixing system. Then for every $h>0$ and every $t<\frac{1}{h},$ one has  $$\dim_H \limsup_{n\to+\infty}B(T^n(x),\frac{1}{n^t})\geq \dim_H D_{\mu}^{\leq}( h).$$
\end{lemme}

The proof of the above lemma is directly derived from the proof of \cite[Lemma 13]{EP}.
\begin{proof}
Fix $\varepsilon>0$ and a Borel measure (which exists by Frostmann's Lemma) $\nu$ such that $\underline{\dim}_H \nu \geq \dim_H E_{\mu}(\leq h)-\varepsilon$ and  $$\nu(E_{\mu}(\leq h))=1.$$
Let $S_h$ be a Borel set with $\nu_h(S_h)=1$ and for every  $y\in E_h,$ $$\underline{\dim}(\mu,y)\leq h.$$
By Lemma \ref{LocdimMix}, for $\mu$-almost every $x$,  one has $$y\in \limsup_{n\to+\infty}B(T^n(x),\frac{1}{n^t}).$$
This yields
\begin{align*}
&\int \mu\left\{x : \ y\in \limsup_{n\to+\infty}B(T^n(x),\frac{1}{n^t})\right\}d\nu(y)=1 \\
&\Leftrightarrow \int \int \chi_{\limsup_{n\to+\infty}B(T^n(x),\frac{1}{n^t})}(y)d\mu(x)d\nu(y)=1 \\
&\Leftrightarrow \int \int \chi_{\limsup_{n\to+\infty}B(T^n(x),\frac{1}{n^t})}(y)d\nu(y)d\mu(x)=1 \\ &\Leftrightarrow \int \int \nu(\limsup_{n\to+\infty}B(T^n(x),\frac{1}{n^t}))d\mu(x)=1
\end{align*}
so that for $\mu$-almost every $x$, one has $$\nu(\limsup_{n\to+\infty}B(T^n(x),\frac{1}{n^t}))=1,$$
hence $$\dim_H \limsup_{n\to+\infty}B(T^n(x),\frac{1}{n^t})\geq \overline{\dim}_H \nu \geq \dim_H D^{\leq}_{\mu}(h)-\varepsilon.$$
Letting $\varepsilon$ tends to $0$ along a countable sequence concludes the proof.
\end{proof}

%
%
%
%
%
%
%

\section{Proof of Theorem \ref{mainthm}} 
In this section, we fix a $\Sigma$-mixing ergodic  system $(T,\mu)$  (as stated in  Definition \ref{defphimix}) such that $\mu$ is exact-dimensional and a non-increasing  sequence of radii $s_{\underline{r}}=(r_n)_{n\in \mathbb{N}}$ such that $s_{\underline{r}}\leq \dim \mu$, where $$s_{\underline{r}}=\inf\left\{s:\sum_{n\geq 1}r_n^s<+\infty\right\}.$$

Note that, by definition of $s_{\underline{r}},$ one has for every $x,$ $$\dim_H E(x,s_{\underline{r}})\leq s_{\underline{r}}.$$

Let us first introduce some notations and make some observation.  

Fix $\varepsilon>0$ and write
\begin{equation}
\label{defdeltaepsi}
\begin{cases} \alpha=\dim \mu\\ \delta_{\varepsilon} =\frac{\alpha+\varepsilon}{s_{\underline{r}}-\varepsilon} \\ s_{\varepsilon}=\frac{(\alpha-\varepsilon)}{\delta_{\varepsilon}}\times \frac{\alpha -\varepsilon}{\alpha+\varepsilon}-\varepsilon.\end{cases}
\end{equation}
One easily sees that $\delta_{\varepsilon} \geq 1$ and by definition of $s_{\underline{r}}$, one has 
\begin{equation}
\label{sumdivmes}
\sum_{n\geq 1} (r_n^{\frac{1}{\delta_\varepsilon}})^{\alpha+\varepsilon}=\sum_{n\geq 1}r_n^{s_{\underline{r}}-\varepsilon} =+\infty.
\end{equation}
 In order to prove the result, the strategy is to prove that \eqref{sumdivmes} implies a very strong full $\mu$-measure statement for the  sequence of balls $\limsup_{n\to +\infty} B(T^n(x),r_n^{\frac{1}{\delta_{\varepsilon}}})$ and to deduce from this property and an appropriate use of the mass transference principle for finite measures (established in \cite{ED3}) that $$\dim_H  \limsup_{n\to +\infty} B(T^n(x),r_n)\geq \frac{\alpha}{\delta_{\varepsilon}}=s_{\varepsilon}.$$
 Letting $\varepsilon\to 0$ would then yield the result. 
 
\subsection{Recall on the mass transference principle for finite measure}

We start this section by recalling the mass transference principle for finite measure. This theorem relies on estimating the so-called essential Hausdorff content associated with measure, which is defined as below.
\begin{definition}[\cite{ED3}]
\label{mucont}
{Let $\mu \in\mathcal{M}(\R^d)$, and $s\geq 0$.
The $s$-dimensional $\mu$-essential Hausdorff content of a set $A\subset \mathcal B(\R^d)$ is defined as}
{\begin{equation}
\label{eqmucont}
 \mathcal{H}^{\mu,s}_{\infty}(A)=\inf\left\{\mathcal{H}^{s}_{\infty}(E): \ E\subset  A , \ \mu(E)=\mu(A)\right\}.
 \end{equation}}
\end{definition}
The mass transference principle for finite measure established in \cite[Theorem 2.2]{ED3} is the following.

\begin{theoreme}[\cite{ED3}]
\label{zzani}
Let $\mu\in\mathcal{M}(\mathbb{R}^d)$ be a probability measure and let $\Big(B_n\Big)_{n\in\mathbb{N}}$  be a sequence of balls such that $\vert B_n \vert \to 0$ and $$\mu\Big(\limsup_{n\to+\infty}\frac{1}{2}B_n\Big)=1.$$
Let $\delta \geq 1$ be a real number. If there exists  $s<\dim\mu,$  such that for every $n\in \mathbb{N}$, one has $$\mathcal{H}^{\mu,s}_{\infty}(\widering{ B_n}^{\delta})\geq \mu(B_n),$$
then $$\dim_H \limsup_{n\rightarrow+\infty}B_n^{\delta}\geq s.$$
\end{theoreme}

 \subsection{Full measure statement for every induced map}
  In the case where $\mu$ is $\alpha$-Alfhors regular, a simple full-measure statement would be sufficient to derive  Theorem \ref{mainthm} (in combination with the classical mass transference principle \cite{BV}). The difficulty here is that the traditional bound $\frac{\dim \mu}{\delta_{\varepsilon}}$ of the mass transference principle does not hold under the only hypothesis that $\mu$ is exact-dimensional\footnote{Actually, the situation is even worst here. Indeed, it can be proved that this uniform bound (on the sequences of balls $(B_n)_{n\in\mathbb{N}}$) does not hold when $\mu$ is a self-affine measure on a Bedford-McMullen carpet in general. }.

 Let us state start by introducing the following classical quantity, which are well-defined as a consequence of Poincaré's recurrence theorem.
 \begin{definition}
\label{definduced}
Let $A$ be a Borel set such that $\mu(A)>0.$ Given $x\in\mathbb{R}^d,$ we define:
\smallskip
\begin{itemize}
\item[(1)] the first hitting time of $x$ in $A$ as 
$$n_A(x)=\inf\left\{n\geq 1 : \ T^n(x)\in A\right\}\in\mathbb{N}\cup \left\{+\infty\right\},$$
\item[(2)] the map $T_A$ induced by $T$ on $A$ as $$T_A(x)=T^{n_A(x)}(x),$$
\item[(3)] for $k\in\mathbb{N},$ the $k$th hitting time of $x$ in $A$ as $$n_A^{k}(x)=\sum_{i=0}^{k-1}n_A(T_A^i(x)).$$
\end{itemize}
\end{definition}
\begin{remark}
\smallskip
\item[•] By Poincaré's recurrence Theorem, for $\mu$-a.e. $x$, $n_A(x)<+\infty.$\medskip
\item[•] It is classical that, defining $\mu_A(\cdot):=\frac{\mu(A\cap \cdot)}{\mu(A)},$ $(T_A,\mu_A)$ is also ergodic and $\mu_A$ is exact-dimensional of dimension $\alpha=\dim(\mu)$.\medskip
\item[•] The $k$th hitting time of $x$ in $A$ is the integer $n_A^k(x)$ such that $T_A^k(x)=T^{n_A^k(x)}(x).$
\end{remark}
 We are now ready to state the main theorem of this section.
\begin{theoreme}
\label{stronfullmes}
Let $(T,\mu)$ be a $\Sigma$-mixing ergodic system, with  exact-dimensional $\mu$. Let $(\widetilde{r}_n)_{n\in\mathbb{N}}$ be a decreasing sequence of radii. If there exists $t>\dim \mu$ such that $$\sum_{n\geq 1} \widetilde{r}_n^t=+\infty,$$
then for every $A$ with $\mu(A)>0,$ for $\mu$-almost every $x$ , $$\mu_A\Big(\limsup_{k\to+\infty}B\Big(T^{k}_A(x),\widetilde{r}_{n_A^k(x)}\Big)\Big)=1.$$ 
\end{theoreme} 
\begin{remark}
With our previous notations, we will eventually take $\widetilde{r}_n=r_n^{\frac{1}{\delta_{\varepsilon}}}$ and $t=\alpha+\varepsilon.$
\end{remark}

The proof of Theorem \ref{stronfullmes} requires some preparation. In the next sub-section, the lemmas necessary to prove Theorem \ref{stronfullmes} are established.

\subsubsection{Some useful lemmas and recalls}

The following lemma is a version of the dynamical Borel-Cantelli lemma. 

\begin{lemme}
\label{thmLLVZ}
Let $(T,\mu)$ be an ergodic measure preserving system. Let $\gamma>0$ be a real number, $E\subset \mathbb{R}^d$ a Borel set such that $\mu(E)=1$ and $\rho:E \to \mathbb{R}_+$ a measurable mapping. Write $\mathcal{C}_1=\mathcal{C}_2=\left\{B(y,r),y\in E,0<r\leq \rho(y)\right\}.$

 Assume that $(T,\mu)$ is $\Sigma$-mixing with respect to $(\mathcal{C}_1,\mathcal{C}_2,\gamma).$ Then for $\mu$-almost every $y$, for every non increasing sequence of radii $(\ell_n)_{n\geq 1}\in \mathbb{R}_+ ^{\mathbb{N}}$ with $\ell_n \to 0,$ 
 \begin{align*}
      \sum_{n\geq 1}\mu(B(y,\ell_n))=+\infty &\Leftrightarrow \mu\Big(\limsup_{n\to +\infty}T^{-n}(B(y,\ell_n))\Big)>0 \\
      &\Leftrightarrow \mu\Big(\limsup_{n\to +\infty}T^{-n}(B(y,\ell_n))\Big)=1.
 \end{align*}

\end{lemme}
Surprisingly it can be proved that the above lemma is actually optimal in a strong sense. Indeed, in \cite{EDcounter}, based on \cite{Saussogalato}, it is established that there exist an  ergodic system $(T,\mathcal{L}^3)$, where $T:\mathbb{T}^3 \to \mathbb{T}^3,$ which is $(\gamma,n\mapsto \frac{C}{n^s})$-mixing with $0<s\leq 1$ which  satisfies that for every $y\in\mathbb{T}^3,$ for every $\theta>\frac{1}{4},$  $$\mathcal{L}^3\Big(\left\{x:T^{n}(x)\in B(y,\frac{1}{n^\theta})\text{ i.o. }\right\}\Big)=0  $$
whereas for every $\frac{1}{\dim_H \mathcal{L}^3}=\frac{1}{3}\geq \theta>\frac{1}{4},$ $$\sum_{n\geq 1}\mathcal{L}^3\Big(B(y,\frac{1}{n^\theta})\Big)=+\infty.$$
We establish Lemma \ref{thmLLVZ}.
\begin{proof}
First, by Borel-Cantelli's Lemma, $$\sum_{n\geq 1}\mu(B(y,\ell_n))<+\infty \Rightarrow \mu\Big(\limsup_{n\to +\infty}T^{-n}(B(y,\ell_n))\Big)=0.$$
Hence, to show the first equivalence, one only needs to prove that $$ \sum_{n\geq 1}\mu(B(y,\ell_n))=+\infty \Rightarrow \mu\Big(\limsup_{n\to +\infty}T^{-n}(B(y,\ell_n))\Big)>0.$$
Let us fix $y\in E $ and write $A_i =T^{-i}(B(y,\ell_i)).$

By Chung-Erdös inequality, for every $Q\leq N,$ one has $$\mu\Big(\bigcup_{Q\leq i\leq N}A_i\Big)\geq \frac{\Big(\sum_{Q\leq i\leq N}\mu(A_i)\Big)^2}{\sum_{Q\leq i,j\leq N}\mu(A_i \cap A_j)}.$$
In particular, if one shows that there exists $C\geq 1$ and $C'>0$ such that 
\begin{equation}
\label{erta}
 \sum_{Q\leq i,j\leq N}\mu(A_i \cap A_j)\leq C'  \sum_{Q\leq i\leq N}\mu(A_i)+C \Big(\sum_{Q\leq i\leq N}\mu(A_i)\Big)^2,
\end{equation}
then, since $\sum_{i\geq 1}\mu(A_i)=+\infty$, letting $N \to +\infty$ yields $$\mu\Big(\bigcup_{ i\geq Q}A_i\Big)\geq \frac{1}{C}.$$ 
Letting $Q\to +\infty$ would then proves the claim. So let us prove that \eqref{erta} holds.  
\begin{align*}
&\sum_{Q\leq i,j\leq N}\mu(A_i \cap A_j)=\sum_{Q\leq i\leq N}\mu(B(y,\ell_i))+2\sum_{Q\leq i<j\leq N} \mu(B(y,\ell_i)\cap T^{-(i-j)}(B(y,\ell_j))) \\
&\leq \sum_{Q\leq i\leq N}\mu(B(y,\ell_i))+2\sum_{Q\leq i<j\leq N} \gamma\mu(B(y,\ell_i))\times \mu(B(y,\ell_j)) +\phi(j-i)\mu(B(y,\ell_j)) \\
&\leq (1+\sum_{n\geq 1}\phi(n))\sum_{Q\leq i\leq N}\mu(B(y,\ell_i))+\gamma\Big(\sum_{Q\leq i\leq N} \mu(B(y,\ell_i))\Big)^2.
\end{align*}
Thus the claim is proved.
Finally, notice that
$$T^n(x)\in B(y,\ell_n)\implies T^{n-1}(T(x))\in B(y,\ell_{n-1})$$
so that
$$x\in \limsup_{n\to +\infty}T^{-n}(B(y,\ell_n)) \Rightarrow T(x)\in \limsup_{n\to +\infty}T^{-n}(B(y,\ell_n)).$$
Assuming that $$\mu\Big(\limsup_{n\to +\infty}T^{-n}(B(y,\ell_n))\Big)>0,$$
by Birkhoff's ergodic theorem, there must exist $x\in \limsup_{n\to +\infty}T^{-n}(B(y,\ell_n))$ such that $$\lim_{p\to +\infty}\frac{\sum_{1\leq k\leq p-1}\chi_{\limsup_{n\to +\infty}T^{-n}(B(y,\ell_n))}(T^k(x))}{p}= \mu\Big(\limsup_{n\to +\infty}T^{-n}(B(y,\ell_n))\Big).$$
But since $x\in \limsup_{n\to +\infty}T^{-n}(B(y,\ell_n)) \implies T^{k}(x) \in \limsup_{n\to +\infty}T^{-n}(B(y,\ell_n))$ for every $k\in\mathbb{N},$ one has $$\mu\Big(\limsup_{n\to +\infty}T^{-n}(B(y,\ell_n))\Big)=1,$$
which proves the last equivalence.
\end{proof}

We are now ready to prove Theorem \ref{stronfullmes}.  
\begin{proof}
Let $(T,\mu)$ be an ergodic $\Sigma$-mixing ergodic system with exact-dimensional $\mu$ and let $(\ell_n)_{n\in\mathbb{N}}$ be a non increasing sequence of radii such that for some $t>\dim_H \mu:=\alpha$, one has $$\sum_{n\geq 1}\ell_n^t=+\infty.$$ To prove Theorem \ref{stronfullmes}, one needs to establish that for every Borel set $A$ with $\mu(A)>0,$ for $\mu$-almost every $x,$ one has   $$\mu\Big(\limsup_{n: T^n(x)\in A}B(T^n(x),\ell_n) \cap A\Big)=\mu(A).$$

Let us recall Besicovitch density theorem.
\begin{lemme}[\cite{Be}, Theorem 2]
 \label{densibesi}
 {Let $m\in\mathcal{M}(\R^d)$, $0<\kappa<1$ and $B$ be a Borel set with $m(B)>0.$ For every $r>0$, set}
 \begin{equation} 
 \label{defniv} 
 {B(r) =\left\{x\in B \ : \ \forall \tilde{r}\leq r, \ m(B(x,\tilde{r})\cap B)\geq \kappa m(B(x,\tilde{r}))\right\}} 
 \end{equation}
 {Then }
 \begin{equation}
{ m\left(\bigcup_{r>0}B(r)\right)=m(B).}
 \end{equation}
 \end{lemme}
For $k\in\mathbf{N},$ write $$A_k=\left\{x\in A: \forall r\leq\frac{1}{k}, \  \ \begin{cases}\mu\Big(B(x,r)\cap A\Big)\geq \frac{1}{2}\mu\Big(B(x,r)\Big) \\ \mu\Big(B(x,r)\Big)\geq r^t\end{cases}\right\}.$$ Since $\mu$ is exact $\alpha<t$ exact-dimensional, due to Besicovitch density theorem, one has $$\mu\Big(\bigcup_{k\in\mathbf{N}}A_k\Big)=\mu(A). $$
In addition, for any $k\in\mathbb{N}$, for every $y\in A_k,$ one has 
\begin{align*}
&\sum_{1\leq i,j \leq N} \mu\Big(T^{-i}\Big(B(y,\ell_i)\cap A\Big)\cap T^{-j}\Big(\Big(B(y,\ell_j)\cap A\Big)\Big) \\
&\leq \sum_{1\leq i,j\leq N} \mu\Big(T^{-i}\Big(B(y,\ell_i)\Big)\cap T^{-j}\Big(B(y,\ell_j)\Big)\Big)  
\end{align*}
and for any large enough $Q\in\mathbb{N}$ $$\sum_{Q\leq  n\leq N} \mu\Big(T^{-n}\Big(B(y,\ell_n)\cap A\Big)\Big)\geq \frac{1}{2}\sum_{Q\leq n\leq N} \mu(B(y,\ell_n)) \geq \frac{1}{2}\sum_{Q\leq n\leq N} \ell_n^t .$$
Thus, the same argument in Lemma \ref{thmLLVZ} yields that $$\mu\Big(\limsup_{n\to+\infty}T^{-n}\Big(B(y,\ell_n)\cap A\Big)\Big)>0.$$
Moreover, noticing that $$T^{n}(x)\in B(y,\ell_n)\cap A \implies T^{n-1}(T(x))\in B(y,\ell_{n-1})\cap A,$$
$$x\in \limsup_{n\to+\infty}T^{-n}\Big(B(y,\ell_n)\cap A\Big) \implies T(x)\in \limsup_{n\to+\infty}T^{-n}\Big(B(y,\ell_n)\cap A\Big).$$
By Birkhoff's ergodic's theorem, we conclude that $$\mu\Big(\limsup_{n\to+\infty}T^{-n}\Big(B(y,\ell_n)\cap A\Big)\Big)=1.$$
Since this holds for every $y\in \widetilde{A}:=\bigcup_{k\geq 1}A_k,$ one has 
\begin{align*}
&\int_{\widetilde{A}}  \mu\Big(\limsup_{n\to+\infty}T^{-n}\Big(B(y,\ell_n)\cap A\Big)\Big) d\mu(y)=\mu(A) \Longleftrightarrow \\
&\int_{\widetilde{A}}  \int\chi_{\limsup_{n\to+\infty}T^{-n}\Big(B(y,\ell_n)\cap A\Big)}(x)d\mu(x) d\mu(y)=\mu(A)
\\&\Longleftrightarrow \int_{\widetilde{A}}  \int\chi_{\limsup_{n:T^n(x)\in A}B(T^n(x),\ell_n)}(y)d\mu(x) d\mu(y)=\mu(A) \\
&\Longleftrightarrow \int  \int\chi_{A\cap \limsup_{n:T^n(x)\in A}B(T^n(x),\ell_n)}(y)d\mu(y) d\mu(x)=\mu(A)\\
&\Longleftrightarrow \int\mu\Big(A\cap \limsup_{n:T^n(x)\in A}B(T^n(x),\ell_n)\Big) d\mu(x)=\mu(A)
\end{align*}
which yields that, for $\mu$-almost every $x$, $$\mu\Big(A\cap \limsup_{n:T^n(x)\in A}B(T^n(x),\ell_n)\Big) =\mu(A),$$
which was our statement.
\end{proof}

\subsubsection{Proof of Theorem \ref{stronfullmes}}

\subsection{Applying the mass transference principle for finite measures}

For $\varepsilon>0, $ recall that $s_{\varepsilon}$ is defined by \eqref{defdeltaepsi}. In addition, set  $$\begin{cases}\widetilde{r}_n=r_n^{\frac{1}{\delta_{\varepsilon}}} \\ t_{\varepsilon}=\alpha +\varepsilon\end{cases}$$
We are going to establish the following result.
\begin{theoreme}
\label{AppliMTPexa}
 For $\mu$-almost every $x$, for every $\varepsilon>0,$ one has $$\mu\Big(\limsup_{n: \mathcal{H}^{\mu,s_{\varepsilon}}(\widering{B(T^{n}(x),r_n)}\geq \mu(B(T^{n}(x),\frac{1}{2}\widetilde{r}_n))}B(T^{n}(x),\frac{1}{2}\widetilde{r}_n)\Big)=1.$$ 
\end{theoreme}
Assuming that Theorem \ref{AppliMTPexa} holds true, Theorem \ref{mainthm} readily follows from the mass transference principle for finite measure, Theorem \ref{zzani} and letting $\varepsilon\to 0$ along a countable sequence. So we now prove Theorem \ref{AppliMTPexa}.

For $n\in\mathbb{N},$ write \begin{align}
& A_n=\left\{x\in\mathbb{R}^d\text{ for every }r\leq \frac{1}{n} , \ r^{\alpha+\varepsilon}\leq \mu\Big(B(x,r)\Big)\leq r^{\alpha-\varepsilon}\right\}\nonumber \\
&\widetilde{A}_n =\left\{x\in A_n : \ \text{ for every }r\leq \frac{1}{n},
\mu\Big(B(x,r)\cap A_n\Big)\geq \frac{1}{2}\mu\Big(B(x,r)\Big) \right\}.
\end{align}
\begin{lemme}
\label{doublebesi}
One has $\mu\Big(\bigcup_{n\geq 1} \widetilde{A}_n\Big)=1.$
\end{lemme}
\begin{proof}
Since $\mu$ is $\alpha$-exact dimensional, one has $$\mu\Big(\bigcup_{n\geq 1} A_n\Big)=1.$$
Given $p\geq n,$ write $$A_{n,p}=\left\{x\in A_n : \text{ for every r}\leq \frac{1}{p}, \ \mu\Big(B(x,r)\cap A_n\Big)\geq \frac{1}{2}\mu\Big(B(x,r)\Big)\right\}.$$
Notice that, since $A_n \subset A_p,$ $A_{n,p}\subset \widetilde{A}_p.$ In addition, by Besicovitch density theorem (Lemma \ref{densibesi}), for every $n\in\mathbb{N},$ one has $$\mu\Big(\bigcup_{p\geq n}A_{n,p}\Big)=\mu\Big(A_n\Big).$$
We conclude that $$\mu\Big(\bigcup_{p\geq 1}\widetilde{A}_p\Big)=1.$$
\end{proof}
Given $x\in\mathbb{R}^d$ and $0\leq s\leq d,$ let us denote $$\mathcal{N}_{x,s}:=\left\{n\in\mathbb{N}: \mathcal{H}^{\mu,s}_{\infty}\Big(\widering{B(T^n(x),r_n)}\Big)\geq \mu\Big(B(T^n(x),\widetilde{r}_n)\Big)\right\}.$$

\begin{lemme}
\label{exaesticont}
For $\mu$-almost every $x$, for every $\varepsilon>0,$ one has $$\mu\Big(\limsup_{n\in\mathcal{N}_{x,s_{\varepsilon}}}B(T^n(x),\frac{1}{2}\widetilde{r}_n)\Big)=1.$$
\end{lemme}
\begin{proof}
As a consequence of Theorem \ref{stronfullmes}  applied $(\frac{1}{2}\widetilde{r}_n)_{n\in\mathbb{N}}$ and $t_{\varepsilon}$, for $\mu$-almost every $x$, one has for every $p\in\mathbb{N},$ $$\mu_{A_p}\Big(\limsup_{k\to+\infty}B(T_{A_p}^k(x),\frac{1}{2}\widetilde{r}_{n_{A_p}^k(x)})\Big)=1.$$ 
Fix $0<u\leq 1$ and $p\in\mathbb{N}$ so large that $\mu\Big(\widetilde{A}_p\Big)\geq u.$ Now fix any $y\in \widetilde{A}_p$ and any $r\leq \frac{1}{2p.}$ By definition of $\widetilde{A}_p,$ one has $$\mu(B(y,r))\geq \frac{1}{2} r^{\dim (\mu) +\varepsilon}.$$
Let us fix $0\leq s\leq \dim(\mu) -\varepsilon$, a set $E\subset B(y,r)$ with $\mu(E)=\mu(B(y,r))$ and a covering $(C_n)_{n\in\mathbb{N}}$ of $E$ by balls of radii smaller that $r\leq \frac{1}{2p}.$ Recall that $$\mu(B(y,r)\cap A_p)\geq \frac{1}{2}\mu(B(y,r)).$$
In addition, for any $n\in\mathbb{N}$ such that $C_n \cap A_p \neq \emptyset,$ one has $$\mu(C_n)\leq 2\vert C_n \vert^{\dim(\mu) -\varepsilon}.$$
This, using a convexity argument, yields 
\begin{align*}
\sum_{n\geq 1}\vert C_n \vert^s &\geq \sum_{n : \ C_n \cap A_p \neq \emptyset}\vert C_n \vert^{(\dim(\mu) -\varepsilon)\times \frac{s}{\dim(\mu)-\varepsilon}}\geq \Big(\sum_{n : \ C_n \cap A_p \neq \emptyset} \mu(C_n)\Big)^{ \frac{s}{\dim(\mu)-\varepsilon}}\\
&\geq \mu\Big(E\cap A_p \cap B(y,r)\Big)^{ \frac{s}{\dim(\mu)-\varepsilon}}\geq \frac{1}{2^d}\mu\Big( B(y,r)\Big)^{ \frac{s}{\dim(\mu)-\varepsilon}}\geq \frac{1}{2^{d+1}}r^{s\times \frac{\dim(\mu) +\varepsilon}{\dim(\mu)-\varepsilon}}.
\end{align*} 
Since the above estimate holds for any covering  of $E\subset B(y,r),$ with $\mu(E)=\mu(B(y,r)),$ taking $s=s_{\varepsilon}=\frac{(\alpha-\varepsilon)}{\delta_{\varepsilon}}\times \frac{\dim(\mu) -\varepsilon}{\dim(\mu)+\varepsilon}-\varepsilon,$  one obtains  $$\mathcal{H}^{\mu,s_{\varepsilon}}_{\infty}\Big(B(y,r)\Big)\geq r^{\frac{\alpha-\varepsilon}{\delta_{\varepsilon}}}.$$
Recall that the above estimate holds for every small enough $r$, hence it also holds that $$\mathcal{H}^{\mu,s_{\varepsilon}}_{\infty}\Big(\widering{B(y,r)}\Big)\geq r^{\frac{\alpha-\varepsilon}{\delta_{\varepsilon}}}.$$

In particular, given $x$, applying the above estimates to  $\widering{B(T^{n_A^k(x)}(x),r_{n_A^k(x)})}$ for every $k\in\mathbb{N}$   so large that $\frac{1}{2^d}\leq r_k^{-\varepsilon}$ and $r_{k}^{\frac{1}{\delta_{\varepsilon}}}\leq \frac{1}{p}, $ one has $$\mathcal{H}^{\mu,s_{\varepsilon}}_{\infty}\Big(\widering{B(T^{n_A^k(x)}(x),r_{n_A^k(x)})}\Big)\geq \Big(r_{n_A^k(x)}\Big)^{\frac{\alpha-\varepsilon}{\delta_{\varepsilon}}}\geq \mu\Big(B(T^{n_A^k(x)}(x),\widetilde{r}_{n_A^k(x)})\Big).$$
This proves that $$\mu\Big(\limsup_{n\in\mathcal{N}_{x,s_{\varepsilon}}}B(T^n(x),\frac{1}{2}\widetilde{r}_n)\Big)\geq \mu(A_p)\geq u.$$
Letting $u\to 1$ proves the claim.
\end{proof}

\section{Proof of Theorem \ref{MainthmNonexa}}
\label{sec-MainthmNonexa}

In this section, one will establish  Theorem \ref{MainthmNonexa} under  a slightly weaker mixing assumption, which,  in practice, should be more convenient to use. The result we are proving is the following one. 

\begin{theoreme}
\label{MainthmNonexabis}
Let $(T,\mu)$ be an ergodic system with $\overline{\dim}_H \mu>0$. Assume that there exists a Borel set $E$ with $\mu(E)=1$ and a measurable mapping $\rho:E\to \mathbb{R}_+$ such that, writing $$\mathcal{C}_1=\left\{D_n(x), \ B(x,r) : \ 2^{-n},r \leq \rho(x),x\in E\right\},$$
$(T,\mu)$ is $\phi$-$(\mathcal{C}_1,\mathcal{B}(\mathbb{R}^d),\kappa)$ mixing, where $\phi:\mathbb{N}\to \mathbb{R}_+$ is super-polynomial. Then, for $\mu$-almost every $x$, for every $\delta >\frac{1}{\overline{\dim}_H \mu},$ one has $$\dim_H \left\{y : \ \vert\vert y-T^n(x)\vert\vert_{\infty}\leq \frac{1}{n^{\delta}}\text{ i.o. }\right\}=\frac{1}{\delta}.$$
\end{theoreme}
Note that that Theorem \ref{MainthmNonexabis} readily implies Theorem \ref{MainthmNonexa} by taking $E=\mathbb{R}^d$, $\rho(x)=C,$ where $C>0$ for any $x$.

\subsubsection{Dynamical coverings on induced systems}

\subsection{Scheme of proof of Theorem \ref{MainthmNonexabis}}
\label{Sec-Scheme}

In \cite{JJMS}, the authors highlighted that, to estimate the Hausdorff dimension of random limsup sets, a fruitful strategy is to transfer (using Fubini's theorem) distribution properties of the sequence of random points (or an orbit in our case) around typical points to a property of the random sequence of balls itself. As a starting point  to establish Theorem \ref{MainthmNonexabis}, we will follow this idea. More precisely, we will prove that there exists a Borel set $A$ with $\mu(A)>0$ and a measurable mapping $\theta_n$ such that for $\mu_A$-almost every $y$, for $\mu$-almost every $x$, for $s\approx \frac{1}{\delta},$ writing 
\begin{equation}
    \label{DefThmNxs}
    \mathcal{N}(x,s,\delta)=\left\{n: \ \mathcal{H}^{\mu_A,s}_{\infty}\Big(B(T^{n}(x),\frac{1}{n^\delta})\Big)\geq \mu_A(B(T^n(x),\frac{1}{n^{\theta_n(x)}}))\right\}
\end{equation}

 one has 
\begin{equation}
\label{argumentclé}
y\in\limsup_{n\in\mathcal{N}(x,s,\delta)}B\Big(T^n(x),\frac{1}{2}\frac{1}{n^{\theta_n(x)}}\Big).
\end{equation}

If the following statement holds, then, one has 
\begin{align*}
&\Leftrightarrow\int\int \mu_A\left\{y\in\limsup_{n\in\mathcal{N}(x,s,\delta)}B\Big(T^n(x),\frac{1}{2}\frac{1}{n^{\theta_n(x)}}\Big)\right\}d\mu(x)=1 \\
&\Leftrightarrow \int\int \chi_{\limsup_{n\in\mathcal{N}(x,s,\delta)}B(T^n(x),\frac{1}{2}\frac{1}{n^{\theta_n(x)}})}(y)d\mu_A(y)d\mu(x)=1 \\
&\Leftrightarrow \int\int \chi_{\limsup_{n\in\mathcal{N}(x,s,\delta)}B(T^n(x),\frac{1}{2}\frac{1}{n^{\theta_n(x)}})}(y)d\mu(x)d\mu_A(y)=1 \\
&\Leftrightarrow  \int \mu_A\Big(\limsup_{n\in\mathcal{N}(x,s,\delta)}B(T^n(x),\frac{1}{2}\frac{1}{n^{\theta_n(x)}})\Big)d\mu(x)=1 \\
&\Leftrightarrow  \mu_A\Big(\limsup_{n\in\mathcal{N}(x,s,\delta)}B(T^n(x),\frac{1}{2}\frac{1}{n^{\theta_n(x)}})\Big)=1 \text{ for }\mu-\text{almost every }x.
\end{align*}
See the definition of $\mathcal{N}(x,s,\delta)$ using Theorem \ref{zzani}, would yield that $$\dim_H \limsup_{n\to+\infty}\Big(B(T^{n}(x),\frac{1}{n^\delta})\Big)\geq s\approx \frac{1}{\delta}.$$
In order to establish that \eqref{argumentclé} holds, the simplified scheme of the proof is to prove that, for an infinity of integers $k$, for $\mu$-almost every $x$,
\begin{itemize}
\item[(1)] $T^n(x)\in B(y,2^{-k})$ where $n\approx \frac{1}{\mu_A(B(y,2^{-k}))},$ \medskip
\item[(2)] $\mu_A(B(y,\frac{1}{n^{\delta}}))\approx \frac{1}{n^{\delta \overline{\dim}_H \mu}},$ \medskip
\item[(3)] $\mu_A(B(T^n(x),\frac{1}{n^\delta}))\geq \mu_A(B(y,\frac{1}{n^{\delta}}))\geq  \frac{1}{n^{\delta \overline{\dim}_H \mu}},$ \medskip
\item[(4)] for every  $0\leq s<\underline{\dim}_H \mu_A \approx \overline{\dim}_H \mu_A,$ $$\mathcal{H}^{\mu_A,s}_{\infty}\Big(B(T^n(x),\frac{1}{n^\delta })\Big)\geq \mu_A(B(T^n(x),\frac{1}{n^\delta }))^{\frac{s}{\underline{\dim}_H \mu_A}}\geq \frac{1}{n^{s \delta}},$$
\item[(5)] Writing $\theta_n(x)$ the number such that $$\mu_A(B(T^n(x),\frac{1}{n^{\theta_n(x)}}))\approx \frac{1}{n},$$
one has $$B(y,2^{-k})\subset B(T^n(x),\frac{1}{2n^{\theta_n(x)}}).$$
\end{itemize}
Then taking $s\approx \frac{1}{\delta}$ will allow to conclude the argument.

\subsection{Proof of Theorem \ref{MainthmNonexabis}}

The next lemma determines our choice of the Borel set $A.$

\begin{lemme}
\label{Lemmarestri}
Fix  $\delta>\frac{1}{\overline{\dim}_H \mu}$ and let $\varepsilon>0$ be small enough so that $$\delta \overline{\dim}_H \mu \geq 1+3\varepsilon.$$ Then there exists a Borel set $A$ with $\mu(A)>0$ such that:
\begin{itemize}
\item[(1)] $1\leq \frac{\overline{\dim}_H \mu_A}{\underline{\dim}_H \mu_A}\leq 1+\varepsilon,$ \medskip\medskip
\item[(2)] there exists $r_0>0$ such that for every $x\in \supp(\mu_A),$ for every $r\leq r_0,$ one has $$ \mu_A\Big(B(x,r)\Big)\leq r^{\underline{\dim}_H \mu -\varepsilon},$$
\item[(3)] for every ball $B$, one has $\mu_A\Big(\partial B\Big)=0.$
\end{itemize}
\end{lemme}
\begin{proof}
We now show that items $(1),(2)$ and $(3)$ hold simultaneously  on a set of positive measure.
Fix  $\varepsilon=(d+3)\varepsilon'.$  Let $$A_1=\left\{x: \ \overline{\dim}_H \mu (1-\varepsilon')\leq\underline{\dim}(\mu,x)\leq \overline{\dim}_H \mu\right\}.$$
Then, one has $\mu(A_1)>0.$ In addition, by Besicovitch density theorem (Lemma \ref{densibesi}), for $\mu_{A_1}$ almost every $x$, one has $$\underline{\dim}(\mu,x)=\underline{\dim}(\mu_{A_1},x).$$
So $\mu_{A_1}$ fulfills the first item with $\varepsilon'$. In addition, there exists $\rho_0>0$ and a compact set  $$A_2\subset \left\{x\in A_1 : \  \forall r\leq \rho_0, \ \mu_{A_1}\Big(B(x,r)\Big)\leq r^{\underline{\dim}_H \mu_{A_1}-\varepsilon'}\right\}$$
satisfies $\mu(A_2)\geq \frac{\mu(A_1)}{2}.$ The same argument as below shows that $\mu_{A_2}$ satisfies $$\underline{\dim}_H \mu_{A_1}\leq \underline{\dim}_H \mu_{A_2}\leq \overline{\dim}_H \mu_{A_2}\leq \overline{\dim}_H \mu_{A_1}.$$
Hence, taking $r_0 \leq \rho_0$ small enough so that for every $r\leq r_0,$  
$$2\leq r^{-\varepsilon'},$$
for every $x\in \supp(\mu_{A_2})$ and every $r\leq r_0$ one has $$\mu_{A_2}(B(x,r))\leq \frac{\mu_{A_1}(B(x,r))}{\frac{1}{2}}\leq r^{\underline{\dim}_H \mu_{A_1}-2\varepsilon'}\leq r^{\underline{\dim}_H \mu_{A_2}-3\varepsilon'}.$$
So, $\mu_{A_2}$ fulfills all the items with $3\varepsilon'$ instead of $\varepsilon.$

To deal with the last item,  we choose a finite number $(H_i)_{1\leq i\leq d}$ of affine subspace in the following fashion:
\begin{itemize}
\item[•] \textbf{step 1:} Set $H_d =\mathbb{R}^d.$ If there exists a ball $B \subset \mathbb{R}^d$ such that $\mu_{A_2}\Big(\partial B\Big)>0,$ then there exists a $d-1$ affine subspace  $H$ such that $\mu_{A_2}(H)>0$ (recall that we endowed $\mathbb{R}^d$ with $\vert \vert \ \cdot \ \vert\vert_{\infty}$). In this case, we set $H_{d-1}=H.$ 

If for every ball, $\mu_{A_2}\Big(\partial B\Big)=0,$ then we set for every $1\leq i\leq d,$ $H_i=H_d.$\medskip
\item[•] \textbf{step $2\leq i \leq d-1$:} repeat step $1$ with $\mu_{A_2 \cap H_i}.$\medskip
\end{itemize}
Two scenarios can occur. Either for every ball, one has $$ \mu_{A_2 \cap H_1}\Big(\partial B\Big)=0,$$
in this case we set $A_3=A_2 \cap H_1$. Or $H_1$ is a $1$ dimensional subspace and there exists a ball $B$ such that $$ \mu_{A_2 \cap H_1}\Big(\partial B\Big)>0.$$
But the latest inequality implies that $\mu$ has an atom and hence is purely atomic. This cannot hold since $\overline{\dim}_H \mu>0.$ So the first scenario occurs. In addition, the same argument as before shows that $\mu_{A_2\cap H_1}$ satisfies all the items with $(3+d) \varepsilon'=\varepsilon,$ which concludes the proof.
\end{proof}

\textbf{For the rest of the section, we fix $\delta>\frac{1}{\overline{\dim}_H \mu}$ and $A$ for which $(1)-(3)$ holds.}

Recall that, for $n\geq 0,$
\begin{equation}
\mathcal{D}_{n}=\left\{2^{-n}(k_1,...,k_d)+2^{-n}[0,1)^d, \ (k_1,...,k_d)\in\mathbb{Z}^d\right\}\text{ and }\mathcal{D}=\bigcup_{n\geq 0}\mathcal{D}_{n}.
\end{equation}
By Lemma \ref{Lemmarestri}, for every $(k_1,...,k_d)\in\mathbb{Z}^d,$ one has 
\begin{equation}
\label{mes0boundDya}
\mu_A\Big(\partial(2^{-n}(k_1,...,k_d)+2^{-n}[0,1)^d)\Big)=\mu_A\Big(\partial B(2^{-n}(k_1,...,k_d),2^{-n})\Big)=0.
\end{equation}
We introduce additional notations:
\begin{itemize}

\item[•] for every $n\in\mathbb{N},$  $D_n (x)$ is the dyadic cube of generation $n$ containing $x$, \medskip
\item[•] for every $n\in\mathbb{N},$ one denotes $$\widetilde{\mathcal{D}}_n=\left\{B(2^{-n}(k_1,...,k_d),2^{-n}), \ (k_1,...,k_d)\in\mathbb{Z}^d\right\}\text{ and }\widetilde{\mathcal{D}}=\bigcup_{n\geq 0} \widetilde{\mathcal{D}}_n.$$
\end{itemize}
A direct check shows that the two following facts hold true:
\begin{itemize}
\item[•] for every $\varepsilon>0$, one has $$\sum_{\widetilde{D}\in\widetilde{\mathcal{D}}}\vert \widetilde{D}\vert^{\varepsilon}\mu_A(\widetilde{D})<+\infty ,$$
\item[•] for $\mu_A$-almost every $x$, for every $n\in\mathbb{N},$ there exists a unique $\widetilde{D}\in\widetilde{\mathcal{D}}_n$ such that $x\in \frac{1}{2}\widetilde{D}.$ One will denote $\widetilde{D}_n(x)=\widetilde{D}.$
\end{itemize}
Given $\rho>0,$ let us write
\begin{equation}
S_{\rho}=\left\{y\in E : \ \rho(y)\geq \rho\right\}.
\end{equation}
Since $\mu_A(E)=1,$ one has $$\mu_A\Big(\bigcup_{\rho>0}S_\rho\Big)=1, $$
where $\rho(\cdot)$ is given as in item $(3)$ of Lemma \ref{Lemmarestri}.  In addition, defining 
\begin{equation}
\label{EquaDensiDyad}
\widetilde{S}_{\rho}=\left\{y\in E \cap A : \ \rho(y)\geq \rho\text{ and }\forall r\leq \rho, \forall n\geq \frac{-\log \rho}{\log 2} \begin{cases} \frac{\mu_A\Big(B(y,r)\cap S_{\rho}\Big)}{\mu_A\Big(B(y,r)\Big)}\geq \frac{1}{2} \\
\frac{ \mu\Big(D_n(y)\cap A\Big)}{\mu\Big(D_n(y)\Big)}\geq \frac{1}{2}\end{cases}\right\},
\end{equation}
the argument used in the proof of Lemma \ref{doublebesi} yields that $$\mu_A\Big(\bigcup_{\rho>1}\widetilde{S}_\rho\Big)=1.$$

Given $k\in\mathbb{N}$ and $\widetilde{D}\in\widetilde{\mathcal{D}}_k,$ define

\begin{equation}
\label{Defngamma}
\gamma_{\widetilde{D},\varepsilon,\delta}=\frac{\delta \log \mu_A(\widetilde{D})}{\log \vert \widetilde{D}\vert}(1+\varepsilon)\text{ and }n_{\gamma_{\widetilde{D}}}=\lfloor k\gamma_{\widetilde{D},\varepsilon,\delta}\rfloor+1.
\end{equation}


\begin{lemme}
\label{LemmaEdelta}
For  $\rho>0,$, $y\in \widetilde{S}_\rho,$  and $y'\in \widetilde{D}_k(y'),$ write $$G_{y',k,\rho,\delta,\varepsilon}=\left\{D\in\mathcal{D}_{n_{\gamma_{\widetilde{D}_k(y')}}}: \ D\cap \frac{1}{2}\widetilde{D}_k(y')\cap S_\rho \neq \emptyset,\mu_A(D)\geq \mu_A(D_{n_{\gamma_{\widetilde{D}_k(y')}}}(y'))\right\}$$ and
\begin{equation}
E_{\widetilde{D}_k(y),\delta,\varepsilon,\rho}=\left\{y'\in S_{\rho}\cap \frac{1}{2}\widetilde{D}_k(y) : \mu_A\Big(\bigcup_{D\in G_{y',k,\rho,\delta,\varepsilon}}D\Big)\leq \vert \widetilde{D}_k(y)\vert^{\varepsilon}\mu_A\Big(\widetilde{D}_k(y)\Big)\right\}.
\end{equation}
Then, for $\mu_A$-almost every $y\in \widetilde{S}_\rho,$ there exists $N_y \in\mathbb{N}$ such that for every $k\geq N_y,$  $$y\notin E_{\widetilde{D}_k(y),\delta,\varepsilon,\rho}.$$
\end{lemme}

\begin{proof}
With the notations of Lemma \ref{LemmaEdelta}, it is direct to check that $$\mu_A(E_{\widetilde{D}_k(y),\delta,\varepsilon,\rho})\leq \vert \widetilde{D}_k(y)\vert^{\varepsilon}\mu_A\Big(\widetilde{D}_k(y)\Big).$$
Recall that $$\sum_{\widetilde{D}\in \widetilde{\mathcal{D}}}\vert \widetilde{D}\vert^{\varepsilon}\mu_A\Big(\widetilde{D}\Big)<+\infty.$$
The conclusion follows from Borel Cantelli's Lemma.

\end{proof}

\begin{remark}
\label{remarkrhoy}
Notice that Lemma \ref{LemmaEdelta} implies that for every $\rho>0,$ for $\mu$-almost every $y\in \widetilde{S}_\rho,$ there exists $N_y\in\mathbb{N}$ such that, for every   $k\geq N_y,$ one has $$ \mu_A\Big(\bigcup_{D\in G_{y,k,\rho,\delta,\varepsilon}}D\Big)\geq \vert \widetilde{D}_k(y)\vert^{\varepsilon}\mu_A\Big(\widetilde{D}_k(y)\Big).$$
For any such $y$ and $\rho>0,$ we set 
\begin{align}
\label{DefFdelta}
F_{y,k,\rho,\delta,\varepsilon}=\bigcup_{D\in G_{y,k,\rho,\delta,\varepsilon}}D
\end{align}
and for $r>0,$ $ F_{y,r,\rho,\delta,\varepsilon}= F_{y,\lfloor \frac{\log r}{-\log 2} \rfloor+1,\rho,\delta,\varepsilon}.$
\end{remark}

In order to establish Theorem \ref{mainthm}, one needs to establish a refinement of Proposition \ref{LocdimMix}. We point out that it is to establish the following lemma that one needs $(T,\mu)$ to be super-polynomially mixing. First, for $\rho>0$ and $y\in S_\rho$ for which Remark \ref{remarkrhoy} applies, let us denote, for $r>0$ and $x\in \mathbb{R}^d,$ $$\widetilde{\tau}_r(x,y)=\left\{n\geq 1 : \ T^n(x)\in F_{y,r,\rho,\delta,\varepsilon} \right\}.$$

\begin{lemme}
\label{ModiLocMix}
For every $\rho>0$ and every $y\in\widetilde{ S}_\rho$ such that $\overline{\dim}(\mu_A,y)\leq d$ to which Remark \ref{remarkrhoy} applies, for $\mu$-almost every $x$, one has $$\lim_{r\to 0 } \frac{\log \widetilde{\tau}_r(x,y)}{-\log \mu_A(F_{y,r,\rho,\delta,\varepsilon})}=1.$$
\end{lemme}

\begin{proof}
Notice that, \eqref{EquaDensiDyad}, for every $D\in G_{y,k,\rho,\delta,\varepsilon},$ one has $$\mu_A(D)\geq \frac{\mu(D)}{2 \mu(A)},$$
so that $$\frac{\mu\Big(F_{y,r,\rho,\delta,\varepsilon}\Big)}{2\mu(A)} \leq \mu_A\Big(F_{y,r,\rho,\delta,\varepsilon}\Big)\leq \frac{\mu\Big(F_{y,r,\rho,\delta,\varepsilon}\Big)}{\mu(A)}.$$
This yields that it is enough to show that, for $\mu$-almost every $x$, $$\lim_{r\to 0 } \frac{\log \widetilde{\tau}_r(x,y)}{-\log \mu(F_{y,r,\rho,\delta,\varepsilon})}=1.$$

Recall that for every large enough $r,$ $ F_{y,r,\rho,\delta,\varepsilon}\subset B(y,r).$ Hence the same estimates as in the proof of Proposition \ref{LocdimMix} yields that $$ \liminf_{r\to 0 } \frac{\log \widetilde{\tau}_r(x,y)}{-\log \mu(F_{y,r,\rho,\delta,\varepsilon})}\geq 1.$$
We now establishes that $$ \limsup_{r\to 0 } \frac{\log \widetilde{\tau}_r(x,y)}{-\log \mu(F_{y,r,\rho,\delta,\varepsilon})}\leq 1.$$
Fix $k\geq \max\left\{ N_y, \lfloor \frac{-\log \rho}{\log 2}\rfloor +1 \right\}$ (where $N_y$ is given as in Remark \ref{remarkrhoy}) large enough so that $$\mu(D_n(y))\geq \vert D_n(y)\vert^{2d} $$
for every $n\geq k$. We are going to establish estimates that are similar to the one established in the proof of Proposition \ref{LocdimMix}, only, we need to be more careful along the diagonal $i=j.$ More precisely, we remark first that $$M_{y,2^{-k}}:=\# G_{y,k,\rho,\delta,\varepsilon}$$ 
is polynomial in $r=2^{-k}.$ 
Moreover, by application of Definition \ref{defphimix}, for every $n\in\mathbb{N}$ and every Borel set $\widetilde{A}$, one has $$\mu(T^{-n}(\widetilde{A})\cap F_{y,r,\rho,\delta,\varepsilon})\leq \kappa\mu(\widetilde{A})\mu(F_{y,r,\rho,\delta,\varepsilon})+M_r \phi(n)\mu(\widetilde{A}).$$

Write $$N_r=\lfloor \frac{1}{\mu(F_{y,r,\rho,\delta,\varepsilon})}\rfloor +1$$ Let us fix $\varepsilon'>0$ small enough so that, for every $r$ small enough,  $$\frac{N_r^{1+\varepsilon'}}{N_r^2 \mu(F_{y,r,\rho,\delta,\varepsilon})}\leq r^{\varepsilon'} .$$
Hence, writing for $1\leq i\leq N_r,$ $$A_i=T^{-i}(F_{y,r,\rho,\delta,\varepsilon}),$$
one has
\begin{align}
\label{EquaLipsch}
&\sum_{1\leq i,j\leq N_r} \mu\Big(A_i \cap A_j\Big)= \sum_{1\leq \vert i-j \vert \leq N_r^{\varepsilon'}}\mu(A_i \cap A_j)+\sum_{ \vert i-j \vert > N_r^{\varepsilon'}}\mu(A_i \cap A_j) \nonumber\\
&\leq N_r^{1+\varepsilon'}\mu(F_{y,r,\rho,\delta,\varepsilon})+N_r(N_r+1 )\mu^2(F_{y,r,\rho,\delta,\varepsilon})+\kappa N_r \mu(F_{y,r,\rho,\delta,\varepsilon}) M_r \sum_{N_r^{\varepsilon'}\leq k \leq N_r}\phi(k).
\end{align}
Since $\phi$ is super-polynomial and that $N_r, M_r$ are polynomial in $r$, one has $$M_r \sum_{N_r^{\varepsilon'}\leq k \leq N_r}\phi(k)<C,$$
where $C$ is uniform in $r$. We conclude that $$\sum_{1\leq i,j\leq N_r} \mu\Big(A_i \cap A_j\Big)\leq  N_r^{1+\varepsilon'}\mu(F_{y,r,\rho,\delta,\varepsilon})+N_r(N_r+1 )\mu^2(F_{y,r,\rho,\delta,\varepsilon})+\kappa C N_r \mu(F_{y,r,\rho,\delta,\varepsilon}).$$
The proof of proposition is then concluded by using the same arguments as in the proof of Proposition \ref{LocdimMix}.
\end{proof}

%
%
%
%
%
%
%
The following lemma is established in \cite{SaussolBarreira}).
\begin{lemme}[\cite{SaussolBarreira}]
\label{AlmostDoubLemma}
There exists a constant $C>1$ such that, for $\mu_A$-almost every $y$, for every $\varepsilon'>0$ there exists $r_{y}>0$ such that for every $r\leq r_y,$ one has $$\mu_A(B(y,\frac{1}{10}r))\geq r^{\varepsilon'}\mu_A(B(y,r)).$$
\end{lemme}
We recall that for $\mu_A$-almost every $y$, for every $b\in\mathbb{N},$ writing $D_{b^n}(y)$ the $b$-adic cell of generation $n$ containing $y$, $$\begin{cases}\underline{\dim}(\mu_A,y)=\liminf_{n\to+\infty}\frac{\log \mu_A(D_{b^n}(x))}{-n\log b} \\ \overline{\dim}(\mu_A,y)=\limsup_{n\to+\infty}\frac{\log \mu_A(D_{b^n}(x))}{-n\log b}.\end{cases}$$ 

In consequence, for any $y$ satisfying the above condition (for any $b\in\mathbb{N}$) and such that $\overline{\dim}(\mu_A,y)\leq d,$ fixing $\varepsilon'>0$ and sorting the balls $\left\{B(y,r)\right\}$ into families $$\mathcal{F}_{k,\varepsilon'}=\left\{B(y,r): \ \frac{\log\mu_A(B(y,r))}{\log r}\in [k\varepsilon',(k+1)\varepsilon']\right\},$$
and letting $\varepsilon'\to 0$ one  obtains that 

\begin{equation}
\label{dimlocBoulePeti}
\underline{\dim}(\mu_A,y)=\liminf_{k\to +\infty}\frac{\log \mu_A(D_{n_{\gamma_{\widetilde{D}_k(y)}}}(y))}{\log \vert D_{n_{\gamma_{\widetilde{D}_k(y)}}}(y)\vert}.
\end{equation}

 So let us fix $\rho>0$ and $y\in \widetilde{S}_{\rho}$ such that:
 \begin{itemize}
 \item[•]$\underline{\dim}_H \mu_A \leq\underline{\dim}(\mu_A,y)\leq\overline{\dim}_H(\mu_A)$ and $\overline{\dim}(\mu_A,y)\leq d,$\medskip
 \item[•] Lemma \ref{AlmostDoubLemma} and \eqref{dimlocBoulePeti}  holds for $y$. 
 \end{itemize}
 
Let us fix $\mathcal{K}_y$ an infinite set  of integers so that 
\begin{itemize}
\item[•] for every $k\in\mathcal{K}_y,$ $2^{-k}\leq \min\left\{\rho,r_y\right\}$ (where $r_y$ is as in Lemma \ref{AlmostDoubLemma}),\medskip
\item[•] for every $k\in\mathcal{K}_y,$ one has $$\frac{\log \mu_A(D_{n_{\gamma_{\widetilde{D}_k(y)}}}(y))}{\log \vert D_{n_{\gamma_{\widetilde{D}_k(y)}}}(y)\vert}\leq (1+\varepsilon)\overline{\dim}_H \mu_A \leq (1+\varepsilon)^2 \underline{\dim}_H \mu_A.$$
\end{itemize}

  $k\in\mathbb{N}$ so small that $2^{-k}\leq \min\left\{\rho,r_y\right\}$ (where $r_y$ is as in Lemma \ref{AlmostDoubLemma}) and such that $$\frac{\log \mu_A(D_{n_{\gamma_{\widetilde{D}_k(y)}}}(y))}{\log \vert D_{n_{\gamma_{\widetilde{D}_k(y)}}}(y)\vert}\leq (1+\varepsilon)\overline{\dim}_H \mu_A \leq (1+\varepsilon)^2 \underline{\dim}_H \mu_A.$$ 
By Lemma \ref{ModiLocMix} there exists a Borel set $E$ with $\mu_A(E)=1$, such that, provided that $k$ has been chosen large enough, there exists $n \in [\frac{1}{\mu_A\Big(F_{y,k,\rho,\delta,\varepsilon}\Big)^{1-\varepsilon}},\frac{1}{\mu_A\Big(F_{y,k,\rho,\delta,\varepsilon}\Big)^{1+\varepsilon}}]$ and $T^{n}(x)\in F_{y,k,\rho,\delta,\varepsilon}.$ Recall that $$\mu_A\Big(F_{y,k,\rho,\delta,\varepsilon}\Big)\geq 2^{-k\varepsilon}\mu_A(\widetilde{D}_k(y)),$$ hence, one has $$\Big(\mu_A(\widetilde{D}_k(y))\Big)^{\delta(1-\varepsilon)}\geq \frac{1}{n^{\delta}}\geq \Big(2^{-k\varepsilon}\mu_A(\widetilde{D}_k(y))\Big)^{\delta(1+\varepsilon)}.$$ 
Hence, recalling the definition of $n_{\gamma_{\widetilde{D}_k(y)}},$ one has $$D_{n_{\gamma_{\widetilde{D}_k(y)}}}(T^n(x))\subset B(T^n (x),\frac{1}{n^{\delta}}).$$
Also, recalling that, provided that $k$ has been chosen large enough, one has $$\Big(2^{-k\varepsilon}\mu_A(\widetilde{D}_k(y)\Big)^{1-\frac{\varepsilon}{\underline{\dim}_H \mu_A-\varepsilon}}\geq \mu_A(\widetilde{D}_k(y)) \geq \frac{1}{n^{1-\varepsilon}}.$$

Moreover, since $T^n(x)\in F_{y,k,\rho,\delta,\varepsilon},$ (define by \eqref{DefFdelta}) one has, recalling that $n_{\gamma_{\widetilde{D}_k(y)}}$ is defined by \eqref{Defngamma},   

\begin{align*}
\mu_A\Big(D_{n_{\gamma_{\widetilde{D}_k(y)}}}(T^n(x))\Big)&\geq  \mu_A\Big(D_{n_{\gamma_{\widetilde{D}_k(y)}}}(y)\Big)\geq \Big(\mu_A(\widetilde{D}_k(y))\Big)^{\delta(1+2\varepsilon)(1+\varepsilon)^2\underline{\dim}_H \mu_A} \\
&\geq \Big(\frac{1}{n}\Big)^{\frac{\delta(1+2\varepsilon)(1+\varepsilon)^2\underline{\dim}_H \mu_A}{1-\varepsilon}}.
\end{align*}


In addition, the same computation as in the proof of Lemma \ref{exaesticont} shows that the following result holds true. 

\begin{lemme}
\label{esticontent}
Let $0\leq s<\underline{\dim}_H \mu_A (1-\varepsilon),$ one has $$\mathcal{H}^{\mu_A,s}_{\infty}\Big(B(T^{n}(x),\frac{1}{n^\delta})\Big)\geq \mu_A\Big(B(T^{n}(x),\frac{1}{n^\delta})\Big)^{\frac{s}{(1-\varepsilon)\underline{\dim}_H \mu_A}}\geq \Big(\frac{1}{n}\Big)^{\frac{\delta(1+2\varepsilon)(1+\varepsilon)^2s}{(1-\varepsilon)^2}}.$$
\end{lemme}

Let $\theta_n(x)$ be the measurable mapping defined by $$\mu_A\Big(B(T^{n}(x),\frac{1}{n^{\theta_n(x)}})\Big)=\Big(\frac{1}{n}\Big)^{(\frac{(1-\frac{3\varepsilon}{\underline{\dim}_H \mu_A-\varepsilon})(1-\frac{\varepsilon}{\underline{\dim}_H \mu_A})}{1+\varepsilon}}.$$
Recall that such a definition is made possible by the assumption that every ball's boundary has $0$ measure. In addition, since $$\frac{1}{n}\geq \Big(\mu_A(\widetilde{D}_k(y))2^{-k\varepsilon}\Big)^{1+\varepsilon},$$
recalling that $\mu_A(\widetilde{D}_k(y))\leq \vert\widetilde{D}_k(y) \vert^{\underline{\dim}_H \mu_A-\varepsilon}$ one has 
\begin{align*}
\Big(\frac{1}{n}\Big)^{\frac{1-\frac{\varepsilon}{\underline{\dim}_H \mu_A}}{1+\varepsilon}}\geq \Big(\mu_A(\widetilde{D}_k(y))\Big)\geq 2^{-k\varepsilon}\mu_A(B(y,2^{-k}))
\end{align*}
So that, using similar estimates, one obtains that
\begin{align*}
\Big(\frac{1}{n}\Big)^{(\frac{(1-\frac{3\varepsilon}{\underline{\dim}_H\mu_A-\varepsilon})(1-\frac{\varepsilon}{\underline{\dim}_H \mu_A})}{1+\varepsilon}}\geq \Big(2^{-k\varepsilon}\mu_A(B(y,2^{-k}))\Big)^{1-\frac{3\varepsilon}{\underline{\dim}_H\mu_A-\varepsilon}}\geq 2^{2k\varepsilon}\mu_A(B(y,2^{-k})).
\end{align*}

Since $T^{n}(x)\in \widetilde{D}_k(y)$, and $$\mu_A\Big(B(T^{n}(x),\frac{1}{n^{\theta_n(x)}})\Big)\geq 2^{2k\varepsilon}\times\mu_A(B(y,2^{-k}))\geq \mu_A(B(y,10 \times 2^{-k})),$$
one must have $$B(y,2^{-k})\subset B(T^{n}(x),\frac{1}{2n^{\theta_n(x)}}).$$
In addition, recalling that  using Lemma \ref{esticontent}, with $$s_{\varepsilon}=\frac{1}{\delta}\times \frac{(\frac{(1-\frac{3\varepsilon}{\underline{\dim}_H \mu_A-\varepsilon})(1-\frac{\varepsilon}{\underline{\dim}_H \mu_A})}{1+\varepsilon}}{\frac{(1+2\varepsilon)(1+\varepsilon)^2}{(1-\varepsilon)^2}},$$
one obtains that $$\mathcal{H}^{\mu_A,s_{\varepsilon}}_{\infty}\Big(B(T^{n}(x),\frac{1}{n^\delta})\Big)\geq \mu_A\Big(B(T^{n}(x),\frac{1}{n^{\theta_n(x)}})\Big). $$

Since the above estimates holds for infinitely many $k$, recalling \eqref{DefThmNxs} one obtains that, for $\mu_A$-almost every $y$, for $\mu$-almost every $x$, $$y\in \limsup_{n\in\mathcal{N}_{x,s_{\varepsilon},\delta}}\Big(B(T^n(x),\frac{1}{2}\frac{1}{n^{\theta_n(x)}})\Big).$$
Hence by \eqref{argumentclé} combined with Theorem \ref{zzani}, for $\mu$-almost every $x,$ one has $$\dim_H \limsup_{n\in\mathbb{N}}B\Big(T^n(x),\frac{1}{n^{\delta}}\Big)\geq s_{\varepsilon}.$$
Letting $\varepsilon\to 0$ along a countable sequence shows that for $\mu$-almost every $x$, one has 
$$\dim_H \limsup_{n\in\mathbb{N}}B\Big(T^n(x),\frac{1}{n^{\delta}}\Big)\geq \frac{1}{\delta}.$$
On the other hand, for every $x\in\mathbb{R}^d,$ for every $\varepsilon>0,$ one has $$ \sum_{n\geq 1} \Big\vert B\Big(T^n(x),\frac{1}{n^{\delta}}\Big)\Big\vert^{(1+\varepsilon)\delta}<+\infty $$
which implies that $$\dim_H \limsup_{n\in\mathbb{N}}B\Big(T^n(x),\frac{1}{n^{\delta}}\Big)\leq \frac{1}{\delta}$$
and Theorem \ref{MainthmNonexabis} is proved.

\section{Some perspective}

\subsection{Exponential mixing with respect to Lipschitz observable}

In view of extending the approximation spectrum given by \eqref{equFanShmelTroub} to the $(\times 2,\times 3 )$, working $\Sigma$-mixing system, as stated in Definition \ref{defphimix}, is very natural. In many other contexts, ergodic systems $(T,\mu)$ can be proved to be exponentially-mixing from the space of Lipschitz mapping against itself. Thus it is very natural to wonder whether the proof-techniques of the present article can be extended to these settings. Below we provide some insight on how we can adapt  the proofs of Theorem \ref{mainthm} and Theorem  \ref{MainthmNonexa} to the case of ergodic systems that are super-polynomially mixing with respect to Lipschitz observable. In what follows, we fix $(r_n )_{n\in\mathbb{N}}$ as in Theorem \ref{mainthm} or $(r_n =\frac{1}{n^{\delta}})_{n\in\mathbb{N}},$ $\delta>\frac{1}{\overline{\dim}_H \mu}$ in the context of Theorem \ref{MainthmNonexa}.

First we recall  the following lemma, which is a reformulation of \cite[Lemma 2.2]{C}.

\begin{lemme}[\cite{C}]
\label{LEmmaKae}
Let $\eta \in\mathcal{M}(\mathbb{R}^d)$, $x\in\supp(\eta)$ such that $\overline{\dim}(x,\eta)<+\infty$ and $\varepsilon>0.$ Then there exists $C_{\varepsilon}>0$, depending on $\varepsilon$ and $\overline{\dim}(x,\eta)$ only, such that $$\limsup_{N\to +\infty}\frac{\#\left\{0\leq k\leq N-1 : \ \mu\Big(B(x,2^{-k})\Big)\leq C_{\varepsilon}\mu\Big(B(x,2^{-k+1})\Big)\right\}}{N} \leq \varepsilon.$$ 

\end{lemme}
As $\mu$-almost every $x$ satisfies that $\overline{\dim}(x,\mu)\leq d$, taking $(\ell_n)_{n\in\mathbb{N}}=(r_n)_{n\in\mathbb{N}}$ or $(\ell_n)_{n\in\mathbb{N}}$ as in Lemma \ref{thmLLVZ}, one relatively easily checks that, by considering, if one must, balls $B(T^n(x),\ell_n^{\kappa_n(x)})$, where $1\leq \kappa_n(x)\leq 1+\varepsilon$, one may freely assume that the balls $B(T^n(x),\ell_n)$  are  $C_{\varepsilon}-$doubling, where $C_{\varepsilon}$ depends only on $\varepsilon$ and $d$. 

In addition, for doubling balls, in the case of ergodic systems that are mixing with respect to Lipschitz observable, we have the following.
\begin{lemme}
\label{LEmmaLips}
Let $(T,\mu)$ be a ergodic system and $A,B$ two balls. Assume that $(T,\mu)$ is $\phi$-mixing and that there exists $C\geq 1$ such that, for $D=A,B,$ one has  $$\mu(2D)\leq C \mu(D).$$
Then, for any $n\in\mathbb{N},$ 
$$\mu\Big(T^{-n}(A)\cap B\Big) \leq C^2\mu(A)\times \mu(B)+\frac{4C\phi(n)}{\vert A \vert \times \vert B \vert}.$$ 
\end{lemme}

\begin{proof}
For $D=A,B,$ let us fix a Lipschitz mapping $f_D :\mathbb{R}^d \to \mathbb{R}_+$ such that $$\begin{cases} \chi_D \leq f_D \leq \chi_{2D} \\ \mbox{Lip }(f)\leq \frac{2}{\vert D\vert}. \end{cases} $$
Since $(T,\mu)$ is $\phi$-mixing with respect to  Lipschitz observable, one has 
\begin{align*}
\mu\Big(T^{-n}(A)\cap B\Big)&\leq\int f_A\circ T^{n}(x)f_B(x)d\mu(x)\\
&\leq \int f_A(x)d\mu(x) \times \int f_B(x)d\mu(x) +\frac{4 \phi(n)}{\vert A\vert \times \vert B\vert} \\
&\leq \mu(2A)\times \mu(2B)+\frac{4 \phi(n)}{\vert A\vert \times \vert B\vert} \\
&\leq C^2 \mu(A) \times \mu(B)+\frac{4 \phi(n)}{\vert A\vert \times \vert B\vert}.
\end{align*}

\end{proof}

We, for instance, explain how to adapt Lemma \ref{thmLLVZ} to the case where $(T,\mu)$ is mixing with respect to Lipschitz observable.

Let $(T,\mu)$ be an ergodic system  which is super-polynomially mixing with respect to Lipschitz observable and assume that a ball $B(y,r)$ satisfies $$\mu(B(y,2r))\leq C\mu(B(y,r)) $$
Write again $A_i =T^{-i}(B(y,r))$ $N_r \approx \frac{1}{\mu(B(y,r))}.$  For every $\varepsilon>0$, performing the same computations as \eqref{EquaLipsch}, using Lemma \ref{LEmmaLips} with $A=B=B(y,r),$ one gets 
\begin{align*}
&\sum_{1\leq i,j\leq N_r} \mu(A_i \cap A_j) =\sum_{ \vert i-j \vert \leq N_r^{\varepsilon} }\mu(A_i \cap A_j)+\sum_{\vert i-j \vert > N_r^{\varepsilon}}\mu(A_i \cap A_j)\\
&\leq N_r^{1+\varepsilon}\mu(B(y,r))+CN_r(N_r+1)\mu(B(y,r))^2+ \frac{N_r}{\vert B(y,r)\vert^2} \times \sum_{N_r^{\varepsilon}\leq k \leq N_r}\phi(k)\\
&N_r^{1+\varepsilon}\mu(B(y,r))+CN_r(N_r+1)\mu(B(y,r))^2+ \frac{N_r \mu(B(y,r))}{\mu(B(y,r))\vert B(y,r)\vert^2} \times \sum_{N_r^{\varepsilon}\leq k \leq N_r}\phi(k).
\end{align*}
Since, $\phi(n)$ decays super-polynomially fast and  $N_r,$ $ \frac{1}{\mu(B(y,r))\vert B(y,r)\vert^2} $ are polynomial in $r$, there exists a constant $\widetilde{C} $ such that $$ \frac{1}{\mu(B(y,r))\vert B(y,r)\vert^2}\times \sum_{N_r^{\varepsilon}\leq k \leq N_r}\phi(k)<\widetilde{C},$$
thus the same arguments as in  \eqref{EquaLipsch} shows that Lemma \ref{thmLLVZ} holds.

From the above argument, it is not complicated to adapt the mixing arguments used to prove Theorem \ref{mainthm} and Theorem \ref{MainthmNonexa}  if one  considers only balls that are doubling with a uniform constant which we can do using Lemma \ref{LEmmaKae}.

\subsection{Dynamical approximation and potential theory}
A careful reader may notice that the proofs of Theorem \ref{mainthm} and Theorem \ref{MainthmNonexa} appear to be somewhat different from the proof of \eqref{resultRandomnonexa} obtained  for random i.i.d. sequences. Moreover, we would like to point out that it was established by Ekström and Persson that \eqref{resultRandomnonexa} holds for measures $\mu\in\mathcal{M}(\mathbb{R}^d)$ that are exact-dmensional whereas its dynamical analog  \eqref{equaPersson} (which was established later on in \cite{PerssInhomo}  by the second author) does not cover this important case. One of the reason for which the arguments used in random settings did not, so far, transfer well to dynamical settings is that it relies on potential theoretical methods. Such methods would lead us to estimate integral of the form $$\int\int_{B_1 \times B_2}f_s(T^{n}(x),T^{k}(y)) d\mu(x)d\mu(y),$$
where $0\leq s\leq d,$ $f_s:\mathbb{R}^2 \to \mathbb{R}$ is a measurable mapping and $B_1,B_2$ are two Borel sets. Thus, in order to make the same argument as in the random case work, the naive hypothesis one would like $(T,\mu)$ to verify is that, there exists $\gamma\geq 1$ and a summable mapping $\phi:\mathbb{N}\to \mathbb{R}_+$ for any pair of Borel sets $(A_1,A_2)$, for every $n\in\mathbb{N},$ one has $$\mu\Big(T^{-n}(A_1)\cap A_2 \Big)\leq \gamma \mu(A_1) \times \mu( A_2)+\phi(n)\mu(A_1).$$
Unfortunately, such a mixing hypothesis is too strong to apply to any interesting example.
\begin{lemme}
Let $(T,\mu)$ be an ergodic system. Then there exists $\gamma \geq 1$ and $\phi:\mathbb{N}\to \mathbb{R}_+$ which satisfy that $\phi(n)\to 0$ such that for every Borel sets $A_1, A_2$ and every $n\in \mathbb{N},$ 
\begin{equation}
\label{MuestDirac}
\mu\Big(T^{-n}(A_1)\cap A_2 \Big)\leq \gamma \mu(A_1) \times \mu( A_2)+\phi(n)\mu(A_1) 
\end{equation}

if, and only if, $\mu =\frac{1}{N}\sum_{k=0}^{N-1}\delta_{T^{k}(x)},$ where $\delta_{\cdot}$ denotes the Dirac mass and $x$ is periodic of period $N$.
\end{lemme}

\begin{proof}
Assume that $(T,\mu)$ satisfies \eqref{MuestDirac} for some $\gamma,\phi.$ Then for every $n\in\mathbb{N},$ one has 
\begin{align*}
\mu\Big(T^{-1}(A_1)\cap A_2 \Big)=\mu\Big(T^{-n-1}(A_1)\cap T^{-n}(A_2) \Big) \leq \gamma\mu(A_1)\times \mu(A_2)+\phi(n)\mu(A_1).
\end{align*}
Letting $n\to+\infty$ and taking $A_1 =A_2 =A$ yields that, for every Borel set $A$ $$\mu(A)\leq \gamma \mu(A)^2.$$
Assume that there exists $x\in \supp(\mu)$ such that $\mu(\left\{x\right\})=0$ and fix $r>0$ small enough so that $$ 0<\mu\Big(B(x,r)\Big)<\frac{1}{\gamma}.$$ 
One has $$ \mu(B(x,r))\leq \gamma \mu(B(x,r))^2 \Leftrightarrow \frac{1}{\gamma}\leq \mu(B(x,r)),$$
which absurd. Thus $\mu$ has atoms and, being an ergodic probability measure, is a discrete measure carried by a periodic orbit. 

The other direction is left as a small exercise to the reader.
\end{proof}

Thus it seems that potential theoretical techniques developed in \cite{EP,JJMS} are not suited to deal fully with  dynamical settings. But  it    would still be interesting develop a potential theoretical tool able to handle both the dynamical and random cases, as, so far, potential theoretical tools tends to work well when the measure involved is regular enough (see \cite{HPWZ, PerssInhomo} for instance) and there are natural examples of ergodic systems that preserves the Lebesgue measure which do not satisfy \eqref{mainthm}. In particular in \cite{EDcounter}, the dynamical system provided to disproof the conclusion of Theorem \ref{mainthm},  preserves the Lebesgue measure but only relevant  upper-bounds are established, so that computing the exact value of the corresponding dynamical Diophantine sets remains an open problem. 
\bibliographystyle{plain}
\bibliography{bibliogenubi}

\end{document}